\documentclass[a4paper, fullpage,10pt]{article}%
\usepackage{amssymb}
\usepackage{amsmath}
\usepackage{bm}
\usepackage{color}
\usepackage{pdfsync}
\usepackage{natbib}
\usepackage{amsfonts}
\usepackage{graphicx}
\usepackage{subfigure}
\usepackage{float}
\usepackage{breqn}%
\providecommand{\U}[1]{\protect\rule{.1in}{.1in}}
\providecommand{\U}[1]{\protect\rule{.1in}{.1in}}
\newtheorem{theo}{Theorem}[section]

\newtheorem{lem}[theo]{Lemma}

\newtheorem{rem}[theo]{Remark}

\numberwithin{equation}{section}

\begin{document}

\title{Uniform Stability of a Particle Approximation of the Optimal Filter
Derivative\thanks{First version: January 2011. Cambridge University Engineering Department Technical report 
number CUED/F-INFENG/TR.668}}
\author{Pierre Del Moral\thanks{Centre INRIA Bordeaux et Sud-Ouest \& Institut de
Math\'ematiques de Bordeaux , Universit\'e de Bordeaux I, 351 cours de la
Lib\'eration 33405 Talence cedex, France (Pierre.Del-Moral@inria.fr)} , Arnaud
Doucet\thanks{Department of Statistics, University of British Columbia, V6T
1Z4 Vancouver, BC, Canada (arnaud@stat.ubc.ca)} , Sumeetpal S.
Singh\thanks{Department of Engineering, University of Cambridge, Trumpington
Street, CB2 1PZ, United Kingdom (sss40@cam.ac.uk)}}
\maketitle

\begin{abstract}
Sequential Monte Carlo methods, also known as particle methods, are a widely
used set of computational tools for inference in non-linear non-Gaussian
state-space models. In many applications it may be necessary to compute the
sensitivity, or derivative, of the optimal filter with respect to the static
parameters of the state-space model; for instance, in order to obtain maximum
likelihood model parameters of interest, or to compute the optimal controller
in an optimal control problem. In \citet{PDS11} an original particle algorithm to
compute the filter derivative was proposed and it was shown using numerical
examples that the particle estimate was numerically stable in the sense that it
did not deteriorate over time.
In this paper we
substantiate this claim with a detailed theoretical study. $\mathbb{L}_{p}$
bounds and a central limit theorem for this particle approximation of the
filter derivative are presented. It is further shown that under mixing
conditions these $\mathbb{L}_{p}$ bounds and the asymptotic variance
characterized by the central limit theorem are uniformly bounded with respect
to the time index. We demonstrate the performance predicted by theory with
several numerical examples. We also use the particle approximation of the
filter derivative to perform online maximum likelihood parameter estimation
for a stochastic volatility model.

\emph{Some key words}: Hidden Markov Models, State-Space Models, Sequential
Monte Carlo, Smoothing, Filter derivative, Recursive Maximum Likelihood.

\end{abstract}

\section{Introduction\label{sec:intro}}

State-space models are a very popular class of non-linear and non-Gaussian
time series models in statistics, econometrics and information engineering;
see for example \citet{cappe2005}, \citet{Douc01}, \citet{DurbinKoopman2001}.
A state-space model is comprised of a pair of discrete-time stochastic
processes, $\left\{  X_{n}\right\}  _{n\mathbb{\geq}0}$ and $\left\{
Y_{n}\right\}  _{n\geq0}$, where the former is an $\mathcal{X}$-valued
unobserved process and the latter is a $\mathcal{Y}$-valued process which is
observed. The hidden process $\left\{  X_{n}\right\}  _{n\mathbb{\geq}0}$ is
a\ Markov process with initial law $dx\pi_{\theta}\left(  x\right)  $ and time
homogeneous transition law $dx^{\prime}f_{\theta}\left(  \left.  x^{\prime
}\right\vert x\right)  $, i.e.
\begin{equation}
X_{0}\sim dx_{0}\pi_{\theta}\left(  x_{0}\right)  \text{ and }\left.
X_{n}\right\vert \left(  X_{n-1}=x_{n-1}\right)  \sim dx_{n}f_{\theta}\left(
\left.  x_{n}\right\vert x_{n-1}\right)  ,\qquad n\geq1. \label{eq:evol}%
\end{equation}
It is assumed that the observations $\left\{  Y_{n}\right\}  _{n\geq0}%
$\ conditioned upon $\left\{  X_{n}\right\}  _{n\mathbb{\geq}0}$ are
statistically independent and have marginal laws
\begin{equation}
\left.  Y_{n}\right\vert \left(  \left\{  X_{k}\right\}  _{k\geq0}=\left\{
x_{k}\right\}  _{k\geq0}\right)  \sim dy_{n}g_{\theta}\left(  \left.
y_{n}\right\vert x_{n}\right)  . \label{eq:obs}%
\end{equation}
Here $\pi_{\theta}\left(  x\right)  $, $f_{\theta}\left(  \left.  x\right\vert
x^{\prime}\right)  $ and $g_{\theta}\left(  \left.  y\right\vert x\right)  $
are densities with respect to (w.r.t.)\ suitable dominating measures denoted
generically as $dx$ and $dy$. For example, if $\mathcal{X}\subseteq
\mathbb{R}^{p}$ and $\mathcal{Y}\subseteq\mathbb{R}^{q}$ then the dominating
measures could be the Lebesgue measures. The variable $\theta$\ in the
densities are the particular parameters of the model. The set of possible
values for $\theta$, denoted $\Theta$, is assumed to be an open subset of
$\mathbb{R}^{d}$. The model (\ref{eq:evol})-(\ref{eq:obs}) is also often
referred to as a hidden Markov model in the literature \citet{cappe2005}.

For a sequence $\left\{  z_{n}\right\}  _{n\geq0}$ and integers $i$, $j$, let
$z_{i:j}$ denote the set $\left\{  z_{i},z_{i+1},...,z_{j}\right\}  $, which
is empty if $j<i$. Equations (\ref{eq:evol}) and (\ref{eq:obs}) define the law
of $\left(  X_{0:n},Y_{0:n-1}\right)  $ which is given by the measure
\begin{equation}
dx_{0}\pi_{\theta}\left(  x_{0}\right)  \prod\limits_{k=1}^{n}dx_{k}f_{\theta
}\left(  \left.  x_{k}\right\vert x_{k-1}\right)  \prod\limits_{k=0}%
^{n-1}dy_{k}g_{\theta}\left(  \left.  y_{k}\right\vert x_{k}\right)  ,
\label{eq:jointdensity}%
\end{equation}
from which the probability density of the observed process, or likelihood, is
obtained
\begin{equation}
p_{\theta}\left(  y_{0:n-1}\right)  =\int dx_{0}\pi_{\theta}\left(
x_{0}\right)  \prod\limits_{k=1}^{n}dx_{k}f_{\theta}\left(  \left.
x_{k}\right\vert x_{k-1}\right)  \prod\limits_{k=0}^{n-1}g_{\theta}\left(
\left.  y_{k}\right\vert x_{k}\right)  .
\label{eq:decompositionjointdistribution}%
\end{equation}
For a realization of observations $Y_{0:n-1}=y_{0:n-1}$, let $\mathbb{Q}%
_{\theta,n}$\ denote the law of $X_{0:n}$\ conditioned on this sequence of
observed variables, i.e.
\[
\mathbb{Q}_{\theta,n}(dx_{0:n})=\frac{1}{p_{\theta}\left(  y_{0:n-1}\right)  }
\left(  dx_{0}\pi_{\theta}\left(  x_{0}\right)  g_{\theta}\left(  \left.
y_{0}\right\vert x_{0}\right)  \prod\limits_{k=1}^{n-1}dx_{k}f_{\theta}\left(
\left.  x_{k}\right\vert x_{k-1}\right)  g_{\theta}\left(  \left.
y_{k}\right\vert x_{k}\right)  \right)  dx_{n}f_{\theta}\left(  \left.
x_{n}\right\vert x_{n-1}\right)
\]
Let $\eta_{\theta,n}$ denote the time $n$ marginal of $\mathbb{Q}_{\theta,n}$.
This marginal, which we call the filter, may be computed recursively using
Bayes' formula:
\[
\eta_{\theta,n+1}(dx_{n+1})=\mathbb{Q}_{\theta,n+1}\left(  dx_{n+1}\right)
=\frac{dx_{n+1}\int\eta_{\theta,n}\left(  dx_{n}\right)  g_{\theta}\left(
\left.  y_{n}\right\vert x_{n}\right)  f_{\theta}\left(  \left.
x_{n+1}\right\vert x_{n}\right)  }{\int\eta_{\theta,n}\left(  dx_{n}^{\prime
}\right)  g_{\theta}\left(  \left.  y_{n}\right\vert x_{n}^{\prime}\right)
},\quad n\geq0
\]
and $\eta_{\theta,0}=\pi_{\theta}$ by convention. Except for simple models
such the linear Gaussian state-space model or when $\mathcal{X}$ is a finite
set, it is impossible to compute $p_{\theta}\left(  y_{0:n}\right)  $,
$\mathbb{Q}_{\theta,n}$ or $\eta_{\theta,n}$ exactly. Particle methods have
been applied extensively to approximate these quantities for general
state-space models of the form (\ref{eq:evol})--(\ref{eq:obs}); see
\citet{cappe2005}, \citet{Douc01}.

The particle approximation of $\mathbb{Q}_{\theta,n}$ is the empirical measure
corresponding to a set of $N\geq1$ random samples termed particles, that is
\begin{equation}
\mathbb{Q}_{\theta,n}^{\text{p},N}\left(  dx_{0:n}\right)  =\frac{1}{N}\sum_{i=1}^{N}\delta_{X_{0:n}^{\left(  i\right)
}}\left(  dx_{0:n}\right)  \label{eq:smcJointDen}%
\end{equation}
where $\delta_{z}\left(  dz\right)  $ denotes the Dirac delta mass located at
$z$. This approximation is referred to as the \emph{path space} approximation
\cite{delmoral2004} and it is denoted by the superscript `p'. The particle
approximation of\ $\eta_{\theta,n}$\ is obtained from $\mathbb{Q}_{\theta
,n}^{\text{p},N}$\ by marginalization
\[
\eta_{\theta,n}^{N}(dx_{n})=\frac{1}{N}\sum_{i=1}^{N}\delta_{X_{n}^{\left(
i\right)  }}\left(  dx_{n}\right)  .
\]
These particles are propagated in time using importance sampling\ and
resampling steps; see \citet{Douc01} and \citet{cappe2005} for a review of the
literature. Specifically, $\mathbb{Q}_{\theta,n+1}^{\text{p},N}$\ is the
empirical measure constructed from $N$ independent samples from%
\begin{equation}
\frac{\mathbb{Q}_{\theta,n}^{\text{p},N}\left( dx_{0:n}\right)  dx_{n+1}f_{\theta}\left(  \left.  x_{n+1}\right\vert
x_{n}\right)  g_{\theta}\left(  \left.  y_{n}\right\vert x_{n}\right)  }%
{\int\mathbb{Q}_{\theta,n}^{\text{p},N}\left(  dx_{0:n}\right)  g_{\theta}\left(  \left.  y_{n}\right\vert x_{n}\right)
}.\label{eq:smcJointDenRecursion}%
\end{equation}
It is a well known fact that the particle approximation of $\mathbb{Q}%
_{\theta,n}$\ becomes progressively impoverished as $n$ increases because of
the successive resampling steps \citep{delmoraldoucet2003,ODCE08}. That is,
the number of distinct particles representing the marginal $\mathbb{Q}%
_{\theta,n}^{\text{p},N}(dx_{0:k})$\ for any fixed $k<n$ diminishes as $n$
increases until it collapses to a single particle -- this is known as the
\emph{particle path degeneracy} problem.

The focus of this paper is on the convergence properties of particle methods
which have been recently proposed to approximate the derivative of the
measures $\{\eta_{\theta,n}(dx_{n})\}_{n\geq0}$ w.r.t.\ $\theta=[\theta
_{1},\ldots\theta_{d}]^{\text{T}}\in\mathbb{R}^{d}$:
\[
\zeta_{\theta,n}=\nabla\eta_{\theta,n}=\left[  \frac{\partial\eta_{\theta,n}%
}{\partial\theta_{1}},\ldots,\frac{\partial\eta_{\theta,n}}{\partial\theta
_{d}}\right]  ^{\text{T}}.
\]
(See Section \ref{sec:computingFiltGrad}\ for a definition.) References
\citet{cerou2001} and \citet{tadicdoucet2003} present particle methods which
have a computational complexity that scales linearly with the number $N$ of
particles. It was shown in \citet{PDS11} (see also \citet{PDS09} for a more
detailed numerical study) that the performance of these $\mathcal{O}(N)$
methods, which inherently rely on the particle approximations of
$\{\mathbb{Q}_{\theta,n}\}_{n\geq0}$\ constructed as in
(\ref{eq:smcJointDenRecursion}) above, degraded over time and it was
conjectured that this may be attributed to the particle path degeneracy
problem. In contrast, the alternative method of \citet{PDS05} was shown in
numerical examples to be stable. The method of \citet{PDS05} is a non-standard
particle implementation that avoids the particle path degeneracy problem at
the expense of a computational complexity per time step which is quadratic in
the number of particles, i.e. $\mathcal{O}(N^{2})$; see Section
\ref{sec:computingFiltGrad} for more details. Supported by numerical examples,
it was conjectured in \citet{PDS11} that even under strong mixing assumptions,
the variance of the estimate of the filter derivative computed with the
$\mathcal{O}(N)$ methods increases at least linearly in time while that of the
$\mathcal{O}(N^{2})$ is uniformly bounded w.r.t.\ the time index. This
conjecture is confirmed in this paper. Specifically, we analyze the
$\mathcal{O}(N^{2})$\ implementation of \citet{PDS05} in Section
\ref{sec:stability} and obtain results on the errors of the approximation, in
particular, $\mathbb{L}_{p}$ bounds and a Central Limit Theorem (CLT) are
presented. We show that these $\mathbb{L}_{p}$ bounds and asymptotic variances
appearing in the CLT\ are uniformly bounded w.r.t.\ the time index when the
state-space model satisfies certain mixing assumptions. In contrast, the
asymptotic variance of the $\mathcal{O}(N)$ implementations, which is also
captured through the CLT, is shown to increase linearly. To the best of our
knowledge, these are the first results of this kind.

An important application of our results, which is discussed in detail in
Section \ref{sec:application}, is to the problem of estimating the parameters
of the model (\ref{eq:evol})--(\ref{eq:obs}) from observed data. The estimates
of the model parameters are found by maximizing the likelihood function
$p_{\theta}(y_{0:n})$ with respect to $\theta$ using a gradient ascent
algorithm which relies on the particle\ approximation of the filter
derivative. The results we present in Section \ref{sec:stability} have bearing
on the performance of the parameter estimation algorithm, which we illustrate
with numerical examples in Section \ref{sec:application}. The
Appendix contains the proofs of the main results as well as that of some
supporting auxiliary results. As a final remark,
although the algorithms and theoretical results are
presented for a state-space model, they may be reinterpreted for Feynman-Kac models as well.

\subsection{Notation and definitions}

\label{subsec:notation} We give some basic definitions from probability and
operator semigroup theory. For a measurable space $(E,\mathcal{E})$\ let
$\mathcal{M}(E)$ denote the set of all finite signed measures\ and
$\mathcal{P}(E)$ the set of all probability measures on $E$. The $n$-fold
product space $E\times\cdots\times E$\ is denoted by $E^{n}$. Let
$\mathcal{B}(E)$\ denote the Banach space of all bounded real-valued and
measurable functions $\varphi:E\rightarrow\mathbb{R}$ equipped with the
uniform norm $\Vert\varphi\Vert=\text{sup}_{x\in E}|\varphi(x)|$. For $\nu
\in\mathcal{M}(E)$ and $\varphi\in\mathcal{B}(E)$, let $\nu(\varphi)=\int
~\nu(dx)~\varphi(x)$ be the Lebesgue integral of $\varphi$ w.r.t.\ $\nu$. If
$\nu$\ is a density w.r.t.\ some dominating measure $dx$ on $E$ then,
$\nu(\varphi)=\int dx$ $\nu(x)~\varphi(x)$. We recall that a bounded integral
kernel $M(x,dx^{\prime})$ from a measurable space $(E,\mathcal{E})$ into an
auxiliary measurable space $(E^{\prime},\mathcal{E}^{\prime})$ is an operator
$\varphi\mapsto M(\varphi)$ from $\mathcal{B}(E^{\prime})$ into $\mathcal{B}%
(E)$ such that the functions
\[
x\mapsto M(\varphi)(x):=\int_{E^{\prime}}M(x,dx^{\prime})\varphi(x^{\prime})
\]
are $\mathcal{E}$-measurable and bounded for any $\varphi\in\mathcal{B}%
(E^{\prime})$. The kernel $M$ also generates a dual operator $\nu\mapsto\nu M$
from $\mathcal{M}(E)$ into $\mathcal{M}(E^{\prime})$ defined by
\[
(\nu M)(\varphi):=\nu(M(\varphi)).
\]
Given a pair of bounded integral operators $(M_{1},M_{2})$, we let
$(M_{1}M_{2})$ the composition operator defined by $(M_{1}M_{2})(\varphi
)=M_{1}(M_{2}(\varphi))$.

A Markov kernel is a positive and bounded integral operator $M$ such that
$M(1)\left(  x\right)  =1$ for any $x\in E$. For $\varphi\in\mathcal{B}(E)$,
let
\[
{\mbox{{\rm osc}}(}\varphi)=\sup_{x,x^{\prime}\in E}\left\vert \varphi
(x)-\varphi(x^{\prime})\right\vert
\]
and let
\[
\mbox{Osc}_{1}(E)=\{\varphi\in\mathcal{B}(E):{\mbox{{\rm osc}}(}\varphi
)\leq1\}.
\]
Let $\beta(M)\in\lbrack0,1]$ denote the Dobrushin coefficient of the Markov
kernel $M$ which is defined by the formula \citep[Prop. 4.2.1]{delmoral2004}:
\[
\beta(M):=\sup{\ \{\mbox{{\rm osc}}(M(\varphi))\;;\;\;\varphi\in
\mbox{{\rm Osc}}_{1}(E^{\prime})\}.}
\]
If there exists a positive constant $\rho$ such that the Markov kernel $M$
satisfies
\[
M(x,dz)\geq\rho M(x^{\prime},dz)\;\text{for all}\;\, x,x^{\prime}\in
E\ \,\;\text{then}\;\beta\left(  M\right)  \leq1-\rho.
\]
For two Markov kernels $M_{1},M_{2}$, $\beta(M_{1}M_{2})\leq\beta(M_{1}%
)\beta(M_{2})$.

Given a positive function $G$ on $E$, let $\Psi_{G}:\nu\in\mathcal{P}%
(E)\mapsto\Psi_{G}(\nu)\in\mathcal{P}(E)$ be the probability distribution
defined by
\[
\Psi_{G}(\nu)(dx):=\frac{\nu(dx)G(x)}{\nu(G)}
\]
provided $\infty>\nu(G)>0$. The definitions above also apply if $\nu$\ is a
density and $M$ is a transition density. In this case all instances of
$\nu(dx)$ should be replaced with $dx\nu(x)$\ and $M(x,dx^{\prime})$ by
$dx^{\prime}M(x,x^{\prime})$ where $dx$ and $dx^{\prime}$ is generic notation
for the dominating measures.

It is convenient to introduce the following transition kernels:
\begin{align*}
Q_{\theta,n}(x_{n-1},dx_{n})  &  =g_{\theta}(y_{n-1}|x_{n-1})dx_{n}f_{\theta
}(x_{n}|x_{n-1})=dx_{n}q_{\theta}(x_{n}|x_{n-1}),\quad n>0,\\
Q_{\theta,k,n}(x_{k},dx_{n})  &  =\left(  Q_{\theta,k+1}Q_{\theta,k+2}\cdots
Q_{\theta,n}\right)  (x_{k},dx_{n}),\quad0\leq k\leq n,
\end{align*}
with the convention that $Q_{\theta,n,n}=Id$, the identity operator. Note that
$Q_{\theta,k,n}(1)\left(  x_{k}\right)  $ is the density of the law of
$Y_{k:n-1}$\ given $X_{k}=x_{k}$.\ For $0\leq p\leq n$, define the potential
function $G_{\theta,p,n}$ on $\mathcal{X}$ to be
\begin{equation}
G_{\theta,p,n}(x_{p})=Q_{\theta,p,n}(1)(x_{p})/\eta_{\theta,p}Q_{\theta
,p,n}(1). \label{DGP}%
\end{equation}
Let the mapping $\Phi_{\theta,k,n}:\mathcal{P}(\mathcal{X})\rightarrow
\mathcal{P}(\mathcal{X})$, $0\leq k\leq n$, be defined as follows
\[
\Phi_{\theta,k,n}(\nu)(dx_{n})=\frac{\nu Q_{\theta,k,n}(dx_{n})}{\nu
Q_{\theta,k,n}(1)}.
\]
It follows that $\eta_{\theta,n}=\Phi_{\theta,k,n}(\eta_{\theta,k})$. For
conciseness, we also write $\Phi_{\theta,n-1,n}$ as $\Phi_{\theta,n}$.

A key quantity that facilitates the recursive computation of the derivative of
$\eta_{\theta,n}$ is the following collection of backward Markov transition
kernels:
\begin{equation}
M_{\theta,n}(x_{n},dx_{n-1})=\frac{\eta_{\theta,n-1}(dx_{n-1})q_{\theta}%
(x_{n}|x_{n-1})}{\eta_{\theta,n-1}(q_{\theta}(x_{n}|\cdot))},\quad n>0.
\end{equation}
Their particle approximations are
\begin{equation}
M_{\theta,n}^{N}(x_{n},dx_{n-1})=\frac{\eta_{\theta,n-1}^{N}(dx_{n-1}%
)q_{\theta}(x_{n}|x_{n-1})}{\eta_{\theta,n-1}^{N}(q_{\theta}(x_{n}|\cdot))}.
\label{eq:Mparticle}%
\end{equation}
These backward Markov kernels are convenient for computing certain conditional
expectations and probability measures. In particular, for $\varphi
\in\mathcal{B}(\mathcal{X}^{2})$, we have
\[
\mathbb{E}_{\theta}\left[  \left.  \varphi\left(  X_{n-1},X_{n}\right)
\right\vert y_{0:n-1},x_{n}\right]  =\int M_{\theta,n}(x_{n},dx_{n-1}%
)\varphi\left(  x_{n-1},x_{n}\right)  ,
\]
and the law of $X_{0:n-1}$ given $X_{n}=x_{n}$ and $Y_{0:n-1}=y_{0:n-1}$ is
$M_{\theta,n}(x_{n},dx_{n-1})\cdots M_{\theta,1}(x_{1},dx_{0})$.

Finally, the following two definitions are needed for the CLT of the particle
approximation of the derivative of $\eta_{\theta,n}$. The bounded integral
operator $D_{\theta,k,n}$ from $\mathcal{X}$ into $\mathcal{X}^{n+1}$ is
defined for any $F_{n}\in\mathcal{B}(\mathcal{X}^{n+1})$ by
\begin{equation}
D_{\theta,k,n}(F_{n})(x_{k}):=\int\left(  \prod\limits_{j=k}^{1}M_{\theta
,j}(x_{j},dx_{j-1})\right)  \left(  \prod\limits_{j=k}^{n-1}Q_{\theta
,j+1}(x_{j},dx_{j+1})\right)  F_{n}(x_{0:n}),\quad0\leq k\leq n,
\end{equation}
with the convention that $\prod\emptyset=1$. The
particle\ approximation, $D_{\theta,k,n}^{N}$, is defined to be
\begin{equation}
D_{\theta,k,n}^{N}(F_{n})(x_{k}):=\int\left(  \prod\limits_{j=k}^{1}%
M_{\theta,j}^{N}(x_{j},dx_{j-1})\right)  \left(  \prod\limits_{j=k}%
^{n-1}Q_{\theta,j+1}(x_{j},dx_{j+1})\right)  F_{n}(x_{0:n}).
\end{equation}
To be concise we write
\[
\eta_{\theta,k}(dx_{k})D_{\theta,k,n}(x_{k},dx_{0:k-1},dx_{k+1:n}%
)\quad\text{as}\quad\eta_{\theta,k}D_{\theta,k,n}(dx_{0:n}).
\]
(And similarly for the particle versions.) Although convention dictates that
$\eta_{\theta,k}D_{\theta,k,n}$ should be understood as the measure
$(\eta_{\theta,k}D_{\theta,k,n})(dx_{0:k-1},dx_{k+1:n})$, when we mean
otherwise it should be clear from the infinitesimal neighborhood.
%

\section{Computing the filter derivative\label{sec:computingFiltGrad}}

For any $F_{n}\in\mathcal{B}(\mathcal{X}^{n+1})$, we have
\begin{align}
&  \nabla\mathbb{Q}_{\theta,n}(F_{n})\nonumber\\
&  =\frac{1}{p_{\theta}\left(  y_{0:n-1}\right)  }\int dx_{0:n}\nabla\left(
\pi_{\theta}\left(  x_{0}\right)  \prod\limits_{k=1}^{n}f_{\theta}\left(
\left.  x_{k}\right\vert x_{k-1}\right)  \quad\prod\limits_{k=0}%
^{n-1}g_{\theta}\left(  \left.  y_{k}\right\vert x_{k}\right)  \right)
F_{n}(x_{0:n})\nonumber\\
&  \quad-\frac{1}{p_{\theta}\left(  y_{0:n-1}\right)  }\mathbb{E}_{\theta
}\left\{  \left.  F_{n}(X_{0:n})\right\vert y_{0:n-1}\right\}  \int
dx_{0:n}\nabla\left(  \pi_{\theta}\left(  x_{0}\right)  \prod\limits_{k=1}%
^{n}f_{\theta}\left(  \left.  x_{k}\right\vert x_{k-1}\right)  \quad
\prod\limits_{k=0}^{n-1}g_{\theta}\left(  \left.  y_{k}\right\vert
x_{k}\right)  \right) \nonumber\\
&  =\mathbb{E}_{\theta}\left\{  \left.  F_{n}(X_{0:n})T_{\theta,n}%
(X_{0:n})\right\vert y_{0:n-1}\right\}  -\mathbb{E}_{\theta}\left\{  \left.
F_{n}(X_{0:n})\right\vert y_{0:n-1}\right\}  \mathbb{E}_{\theta}\left\{
\left.  T_{\theta,n}(X_{0:n})\right\vert y_{0:n-1}\right\}
\label{eq:derivativegeneralfunction}%
\end{align}
where
\begin{align}
T_{\theta,n}(x_{0:n})  &  =\sum_{k=0}^{n}t_{\theta,k}(x_{k-1},x_{k}%
)\label{eq:initSmallT}\\
t_{\theta,k}(x_{k-1},x_{k})  &  =\nabla\log\left(  g_{\theta}\left(  \left.
y_{k-1}\right\vert x_{k-1}\right)  f_{\theta}\left(  \left.  x_{k}\right\vert
x_{k-1}\right)  \right)  ,\quad k>0,\\
t_{\theta,0}(x_{-1},x_{0})  &  =t_{\theta,0}(x_{0})=\nabla\log\pi_{\theta
}\left(  x_{0}\right)  .
\end{align}
The first equality in (\ref{eq:derivativegeneralfunction}) follows from the
definition of $\mathbb{Q}_{\theta,n}$\ and interchanging the order of
differentiation and integration.
The interchange is permissible under certain regularity
conditions \citep{Pfl96}; e.g. a sufficient condition would be the main
assumption in Section \ref{sec:stability} under which the uniform stability
results are proved. The second equality follows from a
change of measure, which then permits an importance
sampling based estimator for the derivative of $\mathbb{Q}_{\theta,n}$;
this is the well known score method, e.g. see \citet[Section 4.2.1]{Pfl96}.
For any $\varphi_{n}%
\in\mathcal{B}(\mathcal{X})$, it follows by setting $F_{n}(x_{0:n}%
)=\varphi_{n}(x_{n})$ in (\ref{eq:derivativegeneralfunction}) that
\begin{align*}
&  \nabla\int\eta_{\theta,n}(dx_{n})\varphi_{n}(x_{n})\\
&  =\mathbb{E}_{\theta}\left\{  \left.  \varphi_{n}(X_{n})T_{\theta,n}%
(X_{0:n})\right\vert y_{0:n-1}\right\}  -\mathbb{E}_{\theta}\left\{  \left.
\varphi_{n}(X_{n})\right\vert y_{0:n-1}\right\}  \mathbb{E}_{\theta}\left\{
\left.  T_{\theta,n}(X_{0:n})\right\vert y_{0:n-1}\right\} \\
&  =\int\zeta_{\theta,n}(dx_{n})\varphi_{n}(x_{n})
\end{align*}
where
\begin{equation}
\zeta_{\theta,n}(dx_{n})=\eta_{\theta,n}(dx_{n})\left(  \mathbb{E}_{\theta
}\left[  \left.  T_{\theta,n}\left(  X_{0:n}\right)  \right\vert
y_{0:n-1},x_{n}\right]  -\mathbb{E}_{\theta}\left[  \left.  T_{\theta
,n}\left(  X_{0:n}\right)  \right\vert y_{0:n-1}\right]  \right)  .
\label{eq:predicFiltGrad}%
\end{equation}
We call $\zeta_{\theta,n}$\ the derivative of $\eta_{\theta,n}$.

Given the particle approximation (\ref{eq:smcJointDen}) of $\mathbb{Q}%
_{\theta,n}$, it is straightforward to construct a particle\ approximation of
$\zeta_{\theta,n}$:
\begin{equation}
\zeta_{\theta,n}^{\text{p},N}(dx_{n})=\sum_{i=1}^{N}\frac{1}{N}\left(
T_{\theta,n}(X_{0:n}^{(i)})-\frac{1}{N}\sum_{j=1}^{N}T_{\theta,n}%
(X_{0:n}^{(j)})\right)  \delta_{X_{n}^{\left(  i\right)  }}\left(
dx_{n}\right)  . \label{eq:pathSmcFiltGrad}%
\end{equation}
This approximation is also referred to as the path space method. Such
approximations were implicitly proposed in \citet{cerou2001} and
\citet{tadicdoucet2003} and there are several reasons why this estimate
appears attractive. Firstly, even with the resampling steps in the
construction of $\mathbb{Q}_{\theta,n}^{\text{p},N}$,\ $\zeta_{\theta
,n}^{\text{p},N}$ can be computed recursively.\ Secondly, there is no need to
store the entire ancestry of each particle, i.e. $\left\{  X_{0:n}^{\left(
i\right)  }\right\}  _{1\leq i\leq N}$, and thus the memory requirement to
construct $\zeta_{\theta,n}^{\text{p},N}$\ is constant over time. Thirdly, the
computational cost per time is $\mathcal{O}(N)$. However, as $\mathbb{Q}%
_{\theta,n}^{\text{p},N}$\ suffers from the particle path degeneracy problem,
we expect the approximation $\zeta_{\theta,n}^{\text{p},N}$\ to worsen over
time. This was indeed observed in numerical examples in \citet{PDS11}\ and it
was conjectured that the asymptotic variance (i.e. as $N\rightarrow\infty$) of
$\zeta_{\theta,n}^{\text{p},N}$ for bounded integrands would increase linearly
with $n$ even under strong mixing assumptions. This is now proven in this article.

An alternative particle method to approximate $\{\zeta_{\theta,n}\}_{n\geq0}$
has been proposed in \citet{PDS05, PDS11}. We now reinterpret this method
using the representation in (\ref{eq:predicFiltGrad}) and a different particle
approximation of $\mathbb{Q}_{\theta,n}$\ that avoids the path degeneracy problem.

The measure $\mathbb{Q}_{\theta,n}$ admits the following backward
representation
\[
\mathbb{Q}_{\theta,n}(dx_{0:n})=\eta_{\theta,n}(dx_{n})\prod\limits_{k=n}%
^{1}M_{\theta,k}(x_{k},dx_{k-1})
\]
and the corresponding particle approximation of $\mathbb{Q}_{\theta,n}$\ is
given by
\[
\mathbb{Q}_{\theta,n}^{N}(dx_{0:n})=\eta_{\theta,n}^{N}(dx_{n})\prod
\limits_{k=n}^{1}M_{\theta,k}^{N}(x_{k},dx_{k-1})
\]
where $M_{\theta,k}^{N}$\ was defined in (\ref{eq:Mparticle}). This now gives
rise to the following particle approximation of $\zeta_{\theta,n}$
\citep{PDS05, PDS11}:
\[
\zeta_{\theta,n}^{N}(\varphi_{n})=\int\mathbb{Q}_{\theta,n}^{N}(dx_{0:n}%
)T_{\theta,n}(x_{0:n})\left(  \varphi_{n}(x_{n})-\eta_{\theta,n}^{N}%
(\varphi_{n})\right)
\]
and indeed $\eta_{\theta,n}^{N}(\varphi_{n})=\int\mathbb{Q}_{\theta,n}%
^{N}(dx_{0:n})\varphi_{n}(x_{n})$. It is apparent that $\mathbb{Q}_{\theta
,n}^{N}$\ constructed using this backward method avoids the degeneracy in
paths. It is even possible to compute $\zeta_{\theta,n}^{N}$\ recursively as
detailed in Algorithm 1; since a recursion for $\eta_{\theta,n}$\ is already
available, it is apparent from (\ref{eq:predicFiltGrad}) that what remains is
to specify a recursion for $\mathbb{E}_{\theta}\left[  \left.  T_{\theta
,n}\left(  X_{0:n}\right)  \right\vert y_{0:n-1},x_{n}\right]  $. Let
$\overline{T}_{\theta,n}(x_{n})$\ denote this term, then for $n\geq1$,
\begin{align*}
\overline{T}_{\theta,n}(x_{n})  &  =\mathbb{E}_{\theta}\left[  \left.
T_{\theta,n}\left(  X_{0:n}\right)  \right\vert y_{0:n-1},x_{n}\right] \\
&  =\mathbb{E}_{\theta}\left[  \left.  T_{\theta,n-1}\left(  X_{0:n-1}\right)
\right\vert y_{0:n-1},x_{n}\right]  +\mathbb{E}_{\theta}\left[  \left.
t_{\theta,n}\left(  X_{n-1},X_{n}\right)  \right\vert y_{0:n-1},x_{n}\right]
\\
&  =\int M_{\theta,n}(x_{n},dx_{n-1})\left(  \mathbb{E}_{\theta}\left[
\left.  T_{\theta,n-1}\left(  X_{0:n-1}\right)  \right\vert y_{0:n-2}%
,x_{n-1}\right]  +t_{\theta,n}\left(  x_{n-1},x_{n}\right)  \right) \\
&  =\int M_{\theta,n}(x_{n},dx_{n-1})\left(  \overline{T}_{\theta,n-1}%
(x_{n-1})+t_{\theta,n}\left(  x_{n-1},x_{n}\right)  \right)
\end{align*}
where $\overline{T}_{\theta,0}(x_{0})=t_{\theta,0}(x_{0})$. Algorithm 1
computes $\zeta_{\theta,n}^{N}$\ recursively in time by computing $\left(
\overline{T}_{\theta,n},\eta_{\theta,n}\right)$ and is
initialized with $\overline{T}_{\theta,0}^{(i)}=t_{\theta,0}(X_{0}^{(i)})$
(see (\ref{eq:initSmallT})) where\ $\left\{  X_{0}^{\left(  i\right)
}\right\}  _{1\leq i\leq N}$\ are samples from $\pi_{\theta}(x_{0})$.

\noindent\hrulefill\vspace{-0.35cm}

\noindent

\begin{center}
\textbf{Algorithm 1: A Particle Method to Compute the
Filter Derivative} \label{alg:smcFiltGrad}
\end{center}

$\bullet$\textsf{ Assume at time }$n-1$\textsf{ that approximate samples}
$\left\{  X_{n-1}^{\left(  i\right)  }\right\}  _{1\leq i\leq N}%
$\textsf{\ from }$\eta_{\theta,n-1}$\textsf{\ and approximations }$\left\{
\overline{T}_{\theta,n-1}^{(i)}\right\}  _{1\leq i\leq N}$\textsf{\ of
}$\left\{  \overline{T}_{\theta,n-1}\left(  X_{n-1}^{\left(  i\right)
}\right)  \right\}  _{1\leq i\leq N}$\textsf{\ are available.}

$\bullet$\textsf{ At time }$n$\textsf{, sample }$\left\{  X_{n}^{\left(
i\right)  }\right\}  _{1\leq i\leq N}$\textsf{\ independently from the
mixture}%
\begin{equation}
\frac{\sum_{j=1}^{N}f_{\theta}\left(  x_{n}|X_{n-1}^{(j)}\right)  g_{\theta
}\left(  \left.  y_{n-1}\right\vert X_{n-1}^{(j)}\right)  }{\sum_{j=1}%
^{N}g_{\theta}\left(  \left.  y_{n-1}\right\vert X_{n-1}^{(j)}\right)  }
\label{eq:bootstrap}%
\end{equation}
\textsf{and then compute }$\left\{  \overline{T}_{\theta,n}^{(i)}\right\}
_{1\leq i\leq N}$\textsf{ and} $\zeta_{\theta,n}^{N}$ \textsf{as follows:}
\begin{align}
&  \overline{T}_{\theta,n}^{(i)}=\frac{\sum_{j=1}^{N}\left(  \overline
{T}_{\theta,n-1}^{(j)}+t_{\theta,n}\left(  X_{n-1}^{(j)},X_{n}^{(i)}\right)
\right)  f_{\theta}\left(  X_{n}^{(i)}|X_{n-1}^{(j)}\right)  g_{\theta}\left(
\left.  y_{n-1}\right\vert X_{n-1}^{(j)}\right)  }{\sum_{j=1}^{N}f_{\theta
}\left(  X_{n}^{(i)}|X_{n-1}^{(j)}\right)  g_{\theta}\left(  \left.
y_{n-1}\right\vert X_{n-1}^{(j)}\right)  },\label{eq:Tupdate}\\
&  \zeta_{\theta,n}^{N}(dx_{n})=\frac{1}{N}\sum_{i=1}^{N}\left(  \overline
{T}_{\theta,n}^{(i)}-\frac{1}{N}\sum_{j=1}^{N}\overline{T}_{\theta,n}%
^{(j)}\right)  \delta_{X_{n}^{\left(  i\right)  }}(dx_{n}).
\label{eq:filtGradUpdate}%
\end{align}

Algorithm 1 uses the bootstrap particle filter of
\citet{gordon93}. Note that any SMC implementation of $\{\eta_{\theta
,n}\}_{n\geq0}$\ may be used, e.g. the auxiliary SMC method of \citet{Pitt99}
or sequential importance resampling with a tailored proposal distribution
\citep{Douc01}. It was conjectured in \citet{PDS11} that the asymptotic
variance of $\zeta_{\theta,n}^{N}(\varphi)$ for bounded integrands $\varphi$ is uniformly bounded
w.r.t.\ $n$ under mixing assumptions. This is established in this article.

\section{Stability of the particle estimates\label{sec:stability}}

The convergence analysis of $\zeta_{\theta,n}^{N}$\ (and $\zeta_{\theta
,n}^{\text{p},N}$ for performance comparison) will largely focus on the
convergence analysis of the $N$-particle measures $\mathbb{Q}_{\theta,n}^{N}$
(and correspondingly $\mathbb{Q}_{\theta,n}^{\text{p},N}$) towards their
limiting values $\mathbb{Q}_{\theta,n}$, as $N\rightarrow\infty$, which is in
turn intimately related to the convergence of the flow of particle measures
$\left\{  \eta_{\theta,n}^{N}\right\}  _{n\geq0}$ towards their limiting
measures $\left\{  \eta_{\theta,n}\right\}  _{n\geq0}$. The $\mathbb{L}_{r}$
error bounds and the central limit theorem presented here have been derived
using the techniques developed in~\citet{delmoral2004} for the convergence
analysis of the particle occupation measures $\eta_{\theta,n}^{N}$ . One of
the central objects in this analysis is the local sampling errors defined as
\begin{equation}
V_{\theta,n}^{N}=\sqrt{N}\left(  \eta_{\theta,n}^{N}-\Phi_{\theta,n}%
(\eta_{\theta,n-1}^{N})\right)  \label{eq:localSampErr}%
\end{equation}
The fluctuation and the deviations of these centered random measures can be
estimated using non-asymptotic Kintchine's type $\mathbb{L}_{r}$-inequalities,
as well as Hoeffding's or Bernstein's type exponential
deviations~\citep{delmoral2004,DeR09}. In \citet{DeM00} it is proved that
these random perturbations behave asymptotically as Gaussian random
perturbations; see Lemma \ref{lem:locErrClt} in the Appendix for more details.
In the proof of Theorem \ref{theo:clt} (a supporting theorem) in the Appendix
we provide some key decompositions expressing the deviation of the particle
measures $\mathbb{Q}_{\theta,n}^{N}$ around its limiting value $\mathbb{Q}%
_{\theta,n}$ in terms of the local sampling errors $(V_{\theta,0}^{N}%
,\ldots,V_{\theta,n}^{N})$. These decompositions are key to deriving the
$\mathbb{L}_{r}$-mean error bounds and central limit theorems for the filter derivative.

The following regularity conditions are assumed.

\textbf{(A)} The dominating measures $dx$ on $\mathcal{X}$ and $dy$ on
$\mathcal{Y}$ are finite, and there exist constants $0<\rho,\delta,c<\infty$
such that for all $(x,x^{\prime},y,\theta)\in\mathcal{X}^{2}\times
\mathcal{Y}\times\Theta$, the derivatives of $\pi_{\theta}(x)$, $f_{\theta
}\left(  x^{\prime}|x\right)  $ and $g_{\theta}\left(  y|x\right)  $ with
respect to $\theta$ exists and
\begin{align}
&  \rho^{-1}\leq f_{\theta}\left(  x^{\prime}|x\right)  \leq\rho,\quad
\delta^{-1}\leq g_{\theta}\left(  y|x\right)  \leq\delta,\label{eq:assAeq1}\\
&  \left\vert \nabla\log\pi_{\theta}\left(  x\right)  \right\vert
\vee\left\vert \nabla\log f_{\theta}\left(  x^{\prime}|x\right)  \right\vert
\vee\left\vert \nabla\log g_{\theta}\left(  y|x\right)  \right\vert \leq c.
\label{eq:assAeq2}%
\end{align}
Admittedly, these conditions are restrictive and fail to hold for many models
in practice. (Exceptions would include applications with a compact
state-space.)
However, they are typically made to establish the time uniform stability of
particle approximations of the filter \citep{delmoral2004, cappe2005} as they
lead to simpler and more transparent proofs. Also, we observe that the
behaviors predicted by the Theorems below seem to hold in practice even in
cases where the state-space models do not satisfy these assumptions; see
Section \ref{sec:application}. Thus the results in this paper can be seen to
provide a qualitative guide to the behavior of the particle approximation even
in the more general setting.

For each parameter vector $\theta\in\Theta$, realization of observations
$y=\{y_{n}\}_{n\geq0}$ and particle number $N$, let $(\Omega,\mathcal{F}%
,\mathbb{P}_{\theta}^{y})$\ be the underlying probability space of the random
process $\{(X_{n}^{(1)},\ldots,X_{n}^{(N)})\}_{n\geq0}$ comprised of the
particle system only. Let $\mathbb{E}_{\theta}^{y}$\ the corresponding
expectation operator computed with respect to $\mathbb{P}_{\theta}^{y}$. The
first of the two main results in this section is a time uniform non-asymptotic
error bound.

\begin{theo}
\label{theo:Lperror} Assume (A). For any $r\geq1$, there exists a constant
$C_{r}$ such that for all $\theta\in\Theta$, $y=\{y_{n}\}_{n\geq0}$, $n\geq0$,
$N\geq1$, and $\varphi_{n}\in\mbox{Osc}_{1}(\mathcal{X})$,
\[
\sqrt{N}\mathbb{E}_{\theta}^{y}\left\{  \left\vert \zeta_{\theta,n}%
^{N}(\varphi_{n})-\zeta_{\theta,n}(\varphi_{n})\right\vert ^{r}\right\}
^{\frac{1}{r}}\leq C_{r}
\]

\end{theo}

Let $\{V_{\theta,n}\}_{n\geq0}$ be a sequence of independent centered Gaussian
random fields defined as follows. For any sequence $\{\varphi_{n}\}_{n\geq0}%
$\ in $\mathcal{B}(\mathcal{X})$ and any $p\geq0$, $\{V_{\theta,n}(\varphi
_{n})\}_{n=0}^{p}$ is a collection of independent zero-mean Gaussian random
variables with variances given by
\begin{equation}
\eta_{\theta,n}(\varphi_{n}^{2})-\eta_{\theta,n}(\varphi_{n})^{2}.
\label{eq:localSampErrLimit}%
\end{equation}

\begin{theo}
\label{theo:gradClt}Assume (A). There exists a constant $C<\infty$ such
that for any $\theta\in\Theta$, $y=\{y_{n}\}_{n\geq0}$, $n\geq0$ and
$\varphi_{n}\in\mbox{Osc}_{1}(\mathcal{X})$, $\sqrt{N}\left(  \zeta_{\theta
,n}^{N}-\zeta_{\theta,n}\right)  (\varphi_{n})$ converges in law, as
$N\rightarrow\infty$, to the centered Gaussian random variable
\begin{equation}
\sum_{p=0}^{n}V_{\theta,p}\left(  G_{\theta,p,n}~\frac{D_{\theta
,p,n}(F_{\theta,n}-\mathbb{Q}_{\theta,n}(F_{\theta,n}))}{D_{\theta,p,n}%
(1)}\right)  \label{eq:gradClt}%
\end{equation}
whose variance is uniformly bounded above by $C$\ where
\[
F_{\theta,n}=\left(  \varphi_{n}-\mathbb{Q}_{\theta,n}(\varphi_{n})\right)
\left(  T_{\theta,n}-\mathbb{Q}_{\theta,n}(T_{\theta,n})\right)  .
\]

\end{theo}

The proofs of both these results are in the Appendix.

As a comparison, we quantify the variance of the particle estimate of the
filter derivative computed using the path-based method (see
(\ref{eq:pathSmcFiltGrad}).) Consider the following simplified example that
serves to illustrate the point. Let $g_{\theta}\left(  \left.  y\right\vert
x\right)  =g\left(  \left.  y\right\vert x\right)  $ (that is $\theta
$-independent), $f_{\theta}\left(  \left.  x_{n}\right\vert x_{n-1}\right)
=\pi_{\theta}(x_{n})$, where $\pi_{\theta}$\ is the initial distribution.
(Note that $f_{\theta}$\ in this case satisfies a rephrased version of
(\ref{eq:assAeq1}) under which the conclusion of Theorem \ref{theo:gradClt}%
\ also holds.) Also, consider the sequence of repeated observations
$y_{0}=y_{1}=\cdots$ where $y_{0}$ is arbitrary. Applying Lemma
\ref{lem:pathClt}\ (in the Appendix) that characterizes the limiting
distribution of $\sqrt{N}(\mathbb{Q}_{\theta,n}^{\text{p},N}-\mathbb{Q}%
_{\theta,n})$ to this special case results in $\sqrt{N}(\zeta_{\theta
,n}^{\text{p},N}-\zeta_{\theta,n})(\varphi)$ (see (\ref{eq:pathSmcFiltGrad}%
))\ having an asymptotic distribution which is Gaussian with mean zero and
variance
\[
n\ \times\pi_{\theta}(\overline{\varphi}^{2})\pi_{\theta}^{\prime}\left[
(\nabla\log\pi_{\theta})^{2}\right]  +\pi_{\theta}\left[  \overline{\varphi
}^{2}(\nabla\log\pi_{\theta})^{2}\right]  -\nabla\pi_{\theta}(\varphi)^{2}%
\]
where $\overline{\varphi}=\varphi-\pi_{\theta}(\varphi)$, $\pi_{\theta
}^{\prime}(x)=\pi_{\theta}(x)g\left(  \left.  y_{0}\right\vert x\right)
/\pi_{\theta}\left(  g\left(  \left.  y_{0}\right\vert \cdot\right)  \right)
$. This variance increases linearly with time in contrast to the time bounded
variance of Theorem \ref{theo:gradClt}.

\section{Application to recursive parameter estimation}

\label{sec:application} Being able to compute $\{\zeta_{\theta,n}\}_{n\geq0}$
is particularly useful when performing online static parameter estimation for
state-space models using Recursive Maximum Likelihood (RML) techniques
\citep{legland1997,PDS05,PDS11}; see also \citet{Kantas09} for a general
review of available particle methods based solutions, including Bayesian ones,
for this problem. The computed filter derivative may also be useful in other
areas; e.g. see \cite{CDM08} for an application in control.

\subsection{Recursive Maximum Likelihood}

Let $\theta^{\ast}$ be the true static parameter generating the observed data
$\{y_{n}\}_{n\geq0}$. Given a finite record of observations $y_{0:T}$, the
log-likelihood may be maximized with the following steepest ascent algorithm:
\begin{equation}
\theta_{k}=\theta_{k-1}+\gamma_{k}\left.  \nabla\log p_{\theta}(y_{0:T}%
)\right\vert _{\theta=\theta_{k-1}},\quad k\geq1, \label{eq:MLbyGradient}%
\end{equation}
where $\theta_{0}$ is some arbitrary initial guess of $\theta^{\ast}$,
$\left.  \nabla\log p_{\theta}(y_{0:T})\right\vert _{\theta=\theta_{k-1}}%
$\ denotes the gradient of the log-likelihood evaluated at the current
parameter estimate and $\{\gamma_{k}\}_{k\geq1}$ is a decreasing positive
real-valued step-size sequence, which should satisfy the following
constraints:
\[
\sum_{k=1}^{\infty}\gamma_{k}=\infty,\qquad\sum_{k=1}^{\infty}\gamma_{k}%
^{2}<\infty.
\]
Although $\nabla\log p_{\theta}(y_{0:T})$\ can be computed using
(\ref{eq:increGrad}), the computation cost can be prohibitive for a long data
record since each iteration of (\ref{eq:MLbyGradient}) would require a
complete browse through the $T+1$ data points. A more attractive alternative
would be a recursive procedure in which the data is run through once only
sequentially. For example, consider the following update scheme:
\begin{equation}
\theta_{n}=\theta_{n-1}+\gamma_{n}\left.  \nabla\log p_{\theta}(y_{n}%
|y_{0:n-1})\right\vert _{\theta=\theta_{n-1}} \label{eq:almostOnlineRML}%
\end{equation}
where $\left.  \nabla\log p_{\theta}(y_{n}|y_{0:n-1})\right\vert
_{\theta=\theta_{n-1}}$ denotes the gradient of $\log p_{\theta}%
(y_{n}|y_{0:n-1})$ evaluated at the current parameter estimate; that is upon
receiving $y_{n}$, $\theta_{n-1}$\ is updated in the direction of ascent of
the conditional density of this new observation. Since we have
\begin{equation}
\left.  \nabla\log p_{\theta}(y_{n}|y_{0:n-1})\right\vert _{\theta
=\theta_{n-1}}=\frac{\int dx_{n}\eta_{\theta_{n-1},n}(x_{n})\left.  \nabla
g_{\theta}\left(  \left.  y_{n}\right\vert x_{n}\right)  \right\vert
_{\theta_{n-1}}+\int dx_{n}\left(  \left.  y_{n}\right\vert x_{n}\right)
\zeta_{\theta_{n-1},n}(x_{n})g_{\theta_{n-1}}}{\int dx_{n}\eta_{\theta
_{n-1},n}(x_{n})g_{\theta_{n-1}}\left(  \left.  y_{n}\right\vert x_{n}\right)
}, \label{eq:increGrad}%
\end{equation}
this clearly requires the filter derivative $\zeta_{\theta,n}$. The algorithm
in the present form is not suitable for online implementation as it requires
re-computing the filter and its derivative at the value $\theta=\theta_{n-1}$
from time zero. The RML procedure uses an approximation of (\ref{eq:increGrad}%
) which is obtained by updating the filter and its derivative using the
parameter value $\theta_{n-1}$ at time $n$; we refer the reader to
\citet{legland1997} for details. The asymptotic properties of the RML
algorithm, i.e.\ the behavior of $\theta_{n}$\ in the limit as $n$\ goes to
infinity, has been studied in the case of an i.i.d.\ hidden process by
\citet{titterington1984} and \citet{legland1997} for a finite state-space
hidden Markov model. It is shown in \citet{legland1997} that under regularity
conditions this algorithm converges towards a local maximum of the average
log-likelihood and that this average log-likelihood is maximized at
$\theta^{\ast}$. A particle version of the RML algorithm of
\citet{legland1997}\ that uses Algorithm 1's estimate of $\eta_{\theta,n}$\ is
presented as Algorithm 2.%

\noindent\hrulefill
\vspace{-0.25cm}

\noindent

\begin{center}
\textbf{Algorithm 2: Particle Recursive Maximum Likelihood \label{alg:smcRml}}

\end{center}

$\bullet$\textsf{\ At time }$n-1$\textsf{ we are given }$y_{0:n-1}%
$,\textsf{\ the previous estimate }$\theta_{n-1}$\textsf{\ of }$\theta^{\ast}%
$\textsf{ and }$\{(X_{n-1}^{\left(  i\right)  },\overline{T}_{n-1}%
^{(i)})\}_{i=1}^{N}$\textsf{.}

$\bullet$\textsf{\ At time }$n$\textsf{, upon receiving }$y_{n}$\textsf{,
sample }$\left\{  X_{n}^{\left(  i\right)  }\right\}  _{1\leq i\leq N}%
$\textsf{\ independently from (\ref{eq:bootstrap}) using parameter }%
$\theta=\theta_{n-1}$\textsf{ to obtain }%
\[
\eta_{n}^{N}(dx_{n})=\frac{1}{N}\sum_{i=1}^{N}\delta_{X_{n}^{\left(  i\right)
}}(dx_{n})
\]
\textsf{and then compute}
\begin{align}
&  \overline{T}_{n}^{(i)}=\frac{\sum_{j=1}^{N}\left(  \overline{T}_{n-1}%
^{(j)}+t_{\theta_{n-1},n}\left(  X_{n-1}^{(j)},X_{n}^{(i)}\right)  \right)
f_{\theta_{n-1}}\left(  X_{n}^{(i)}|X_{n-1}^{(j)}\right)  g_{\theta_{n-1}%
}\left(  \left.  y_{n-1}\right\vert X_{n-1}^{(j)}\right)  }{\sum_{j=1}%
^{N}f_{\theta_{n-1}}\left(  X_{n}^{(i)}|X_{n-1}^{(j)}\right)  g_{\theta_{n-1}%
}\left(  \left.  y_{n-1}\right\vert X_{n-1}^{(j)}\right)  }%
,\label{eq:TupdateRML}\\
&  \zeta_{n}^{N}(dx_{n})=\frac{1}{N}\sum_{i=1}^{N}\left(  \overline{T}%
_{n}^{(i)}-\frac{1}{N}\sum_{j=1}^{N}\overline{T}_{n}^{(j)}\right)
\delta_{X_{n}^{\left(  i\right)  }}(dx_{n}), \label{eq:filtGradUpdateRML}%
\end{align}
\textsf{and }
\[
\widehat{\nabla}\log p\left(  \left.  y_{n}\right\vert y_{0:n-1}\right)
=\frac{\int\eta_{n}^{N}(dx_{n})\left.  \nabla g_{\theta}\left(  \left.
y_{n}\right\vert x_{n}\right)  \right\vert _{\theta_{n-1}}+\int\zeta_{n}%
^{N}(dx_{n})g_{\theta_{n-1}}\left(  \left.  y_{n}\right\vert x_{n}\right)
}{\int\eta_{n}^{N}(dx_{n})g_{\theta_{n-1}}\left(  \left.  y_{n}\right\vert
x_{n}\right)  }.
\]
\textsf{Finally update the parameter: }
\begin{equation}
\theta_{n}=\theta_{n-1}+\gamma_{n}\widehat{\nabla}\log p\left(  \left.
y_{n}\right\vert y_{0:n-1}\right)  . \label{eq:SMCrml}%
\end{equation}%
\noindent\hrulefill

Under Assumption A, the particle approximation of the filter is
stable \citep{delmoral2004}; see also Lemma \ref{lem:lpErrorFilter} in the Appendix. This combined with the proven stability of the
particle approximation of the filter derivative implies
that the particle estimate of the derivative of
$\log p\left(  \left. y_{n}\right\vert y_{0:n-1}\right)$ is also stable.

\subsection{Simulations}

The RML algorithm is applied to the following stochastic volatility model
\citep{Pitt99}:
\begin{align*}
X_{0}  &  \sim\mathcal{N}\left(  0,\frac{\sigma^{2}}{1-\phi^{2}}\right)
,\text{ }X_{n+1}=\phi X_{n}+\sigma V_{n+1},\\
Y_{n}  &  =\beta\exp\left(  X_{n}/2\right)  W_{n},
\end{align*}
where $\mathcal{N}\left(  m,s\right)  $\ denotes a Gaussian random variable
with mean $m$ and variance $s$, $V_{n}\overset{\text{i.i.d.}}{\sim}%
\mathcal{N}\left(  0,1\right)  $ and $W_{n}\overset{\text{i.i.d.}}{\sim
}\mathcal{N}\left(  0,1\right)  $ are two mutually independent sequences, both
independent of the initial state $X_{0}$. The model parameters, $\theta
=\left(  \phi,\sigma,\beta\right)  $, are to be estimated.

Our first example demonstrates the theoretical results in Section
\ref{sec:stability}. The estimate of $\partial/\partial\sigma$ $\log p\left(
\left.  y_{n:n+L-1}\right\vert y_{0:n-1}\right)  $ at $\theta^{\ast
}=(0.8,\sqrt{0.1},1)$ \ was computed using Algorithm 1 with 500 particles and
using the path-space method (see (\ref{eq:pathSmcFiltGrad})) with
$2.5\times10^{5}$ particles for the stochastic volatility model. The block
size $L$ was 500. Shown in Figure \ref{fig:blockScroreIncNandN2} is the
variance of these particle estimates for various values of $n$ derived from
many independent random replications of the simulation. The linear increase of
the variance of the path-space method as predicted by theory is evident
although Assumption A is not satisfied.

For the path-space method, because the variance of the estimate of the filter
derivative grows linearly in time, the eventual high variance in the gradient
estimate can result in the divergence of the parameter estimates. To
illustrate this point, (\ref{eq:SMCrml}) was implemented with the path-space
estimate of the filter derivative (\ref{eq:pathSmcFiltGrad}) computed with
10000 particles and constant step-size sequence, $\gamma_{n}=10^{-4}$ for all
$n$. $\theta_{0}$ was initialized at the true parameter value. A\ sequence of
two million observations was simulated with $\theta^{\ast}=(0.8,\sqrt{0.1}%
,1)$. The results are shown in Figure \ref{fig:rml_Nmethod}.

\begin{figure}[ptb]
\centering \includegraphics[scale=0.5]{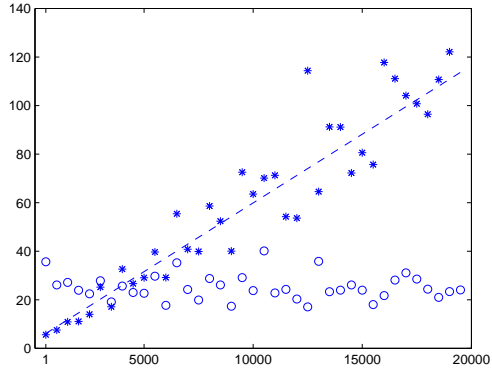}\caption{Variance
of the particle estimates of $\partial/\partial\sigma\log p\left(  \left.
y_{n:n+500-1}\right\vert y_{0:n-1}\right)  $ for various values of $n$ for the
stochastic volatility model. Circles are variance of Algorithm 1's estimate
with 500 particles. Stars indicate the variance of the estimate of the
path-space method with $2.5\times10^{5}$ particles. Dotted line is best
fitting straight line to path-space method's variance to indicate trend.}%
\label{fig:blockScroreIncNandN2}%
\end{figure}

\begin{figure}[ptb]
\centering \includegraphics[scale=0.5]{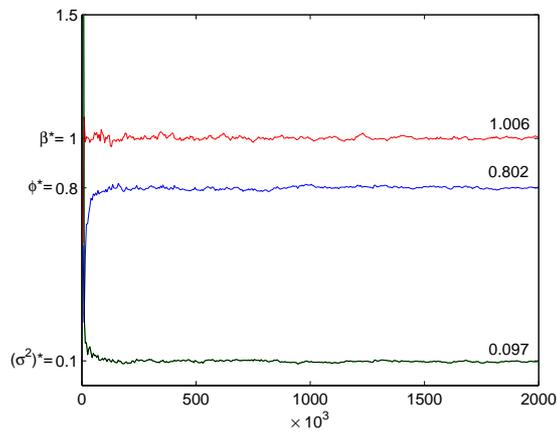}\caption{Sequence
of recursive parameter estimates, $\theta_{n}=\left(  \sigma_{n},\phi
_{n},\beta_{n}\right)  $, computed using (\ref{eq:SMCrml}) with $N=500$. From
top to bottom: $\beta_{n}$, $\phi_{n}$ and $\sigma_{n}$ and marked on the
right are the \textquotedblleft converged values\textquotedblright\ which were
taken to be the empirical average of the last 1000 values. }%
\label{figRecSV}%
\end{figure}

For the same value of $\theta^{\ast}$\ and sequence of observations used in
the previous example, Algorithm 2 was executed with 500 particles and
$\gamma_{n}=0.01$, $n\leq10^{5}$, $\gamma_{n}=(n-5\times10^{4})^{-0.6}$,
$n>10^{5}$. As it can be seen from the results in Figure \ref{figRecSV} the
estimate converges to a value in the neighborhood of the true parameter.

\begin{figure}[ptb]
\centering
\includegraphics[scale=0.5]{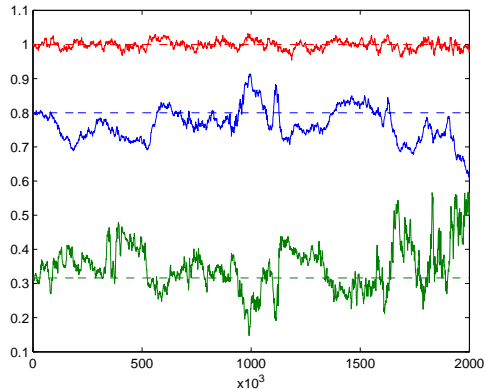}\caption{RML
for stochastic volatility with path-space gradient estimate with 10,000
particles, constant step-size and initialized at the true parameter values
which are indicated by the dashed lines. From top to bottom, $\phi$, $\beta$
and $\sigma$.}%
\label{fig:rml_Nmethod}%
\end{figure}

\section{Conclusion}
We have presented theoretical results establishing the uniform stability of
the particle approximation of the optimal filter derivative proposed in \citet{PDS05, PDS09}. 
While these results have been presented in the context of
state-space models, they can also be applied to Feynman-Kac models \citep{delmoral2004} 
which could potentially enlarge the range of applications. For example,
if $dx^{\prime}f_{\theta}\left(  \left.  x^{\prime}\right\vert x\right)  $ is
reversible w.r.t. to some probability measure $\mu_{\theta}$ and if we replace
$g_{\theta}\left(  \left.  y_{n}\right\vert x_{n}\right)  $ with a
time-homogeneous potential function $g_{\theta}\left(  x_{n}\right)  $ then
$\eta_{\theta,n}$ converges, as $n\rightarrow\infty$, to the probability
measure $\mu_{\theta,h}$ defined as
\[
\mu_{\theta,h}(dx):=\frac{1}{\mu_{\theta}(h_{\theta}\text{ }\int dx^{\prime
}f_{\theta}\left(  \left.  x^{\prime}\right\vert \cdot\right)  h_{\theta
}(x^{\prime}))}\mu_{\theta}(dx)~h_{\theta}(x)~\int dx^{\prime}f_{\theta
}\left(  \left.  x^{\prime}\right\vert x\right)  h_{\theta}(x^{\prime})
\]
where $h_{\theta}$ is a positive eigenmeasure associated with the top
eigenvalue of the integral operator $Q_{\theta}(x,dx^{\prime})=g_{\theta
}(x)dx^{\prime}f_{\theta}\left(  \left.  x^{\prime}\right\vert x\right)  $
(see section 12.4 of~\citet{delmoral2004}). The measure $\mu_{\theta,h}$ is the invariant
measure of the $h$-process defined as the Markov chain with transition kernel
$M_{\theta}\left(  x,dx^{\prime}\right)  \propto dx^{\prime}f_{\theta}\left(
\left.  x^{\prime}\right\vert x\right)  h_{\theta}(x^{\prime})$. The particle
algorithm described here can be directly used to approximate the derivative of
this invariant measure w.r.t to $\theta$. It would also be of interest to
weaken Assumption A and there are several ways this might be approached. For
example for non-ergodic signals using ideas in \citet{OuR05, HeC08}\ or via
Foster-Lyapunov conditions as in \citet{BCJ11, Whi11}.
\section{Acknowledgement}

We are grateful to Sinan Yildirim for carefully reading this report.

\section{Appendix\label{sec:appendix}}

The statement of the results in this section hold for any $\theta$ and any
sequence of observations $y=\{y_{n}\}_{n\geq0}$.
All mathematical expectations are taken with respect to the law of the particle system only for
the specific $\theta$ and $y$ under consideration.
While $\theta$
is retained in the statement of the results, it is omitted in the proofs.
The superscript $y$ of the expectation operator is also omitted in the proofs.

This section commences with some essential definitions in addition to those in Section \ref{subsec:notation}. Let
\[
P_{\theta,k,n}(x_{k},dx_{n})=\frac{Q_{\theta,k,n}(x_{k},dx_{n})}%
{Q_{\theta,k,n}(1)(x_{k})},
\]
and
\[
\mathcal{M}_{\theta,p}(x_{p},dx_{0:p-1})=\prod\limits_{k=p}^{1}M_{\theta
,k}(x_{k},dx_{k-1}),\quad p>0,
\]
and its corresponding particle approximation is
\[
\mathcal{M}_{\theta,p}^{N}(x_{p},dx_{0:p-1})=\prod\limits_{k=p}^{1}%
M_{\theta,k}^{N}(x_{k},dx_{k-1})
\]
To make the subsequent expressions more terse, let
\begin{equation}
\widetilde{\eta}_{\theta,n}^{N}=\Phi_{\theta,n}(\eta_{\theta,n-1}^{N}),\quad
n\geq0, \label{eq:etaTildaDef}
\end{equation}
where $\widetilde{\eta}_{\theta,0}^{N}=\Phi_{\theta,0}(\eta_{-1}^{N}%
)=\eta_{\theta,0}=\pi_{\theta}$ by convention. (Recall $\Phi_{\theta,n}=\Phi_{\theta,n-1,n}$.) Let
\[
\mathcal{F}_{n}^{N}=\sigma\left(  \left\{  X_{k}^{(i)};0\leq k\leq n,1\leq
i\leq N\right\}  \right)  ,\quad n\geq0,
\]
be the natural filtration associated with the $N$-particle approximation model
and let $\mathcal{F}_{-1}^{N}$ be the trivial sigma field.

The following estimates are a straightforward consequence of Assumption (A).
For all $\theta$ and time indices $0\leq k < q\leq n$,
\begin{equation}
b_{\theta,k,n}=\sup_{x_{k},x_{k}^{\prime}}{\frac{Q_{\theta,k,n}(1)(x_{k}%
)}{Q_{\theta,k,n}(1)(x_{k}^{\prime})}}\leq\rho^{2}\delta^{2},\quad\beta\left(
\frac{Q_{\theta,k,q}(x_{k},dx_{q})Q_{\theta,q,n}(1)(x_{q})}{Q_{\theta
,k,q}(Q_{\theta,q,n}(1)) (x_k) }\right)  \leq\left(  1-\rho^{-4}\right)
^{(q-k)}=\overline{\rho}^{q-k}, \label{eq:contractionEst}%
\end{equation}
and for $\theta$, $0<k\leq q$,%
\begin{equation}
M_{\theta,k}^{N}(x,dz)\leq\rho^{4}~M_{\theta,k}^{N}(x^{\prime}%
,dz)\Longrightarrow\beta\left(  M_{\theta,q}^{N}\cdots M_{\theta,k}%
^{N}\right)  \leq\left(  1-\rho^{-4}\right)  ^{q-k+1}.
\label{eq:contractionEst2}%
\end{equation}
Note that setting $q=n$ in (\ref{eq:contractionEst}) yields an estimate for 
$\beta(P_{\theta,k,n})$

Several auxiliary results are now presented, all of which hinge on the
following Kintchine type moment bound proved in \citet[Lem.
7.3.3]{delmoral2004}.

\begin{lem}
\label{lem:khinchine}\citet[Lemma 7.3.3]{delmoral2004}Let $\mu$ be a
probability measure on the measurable space $(E,\mathcal{E})$. Let
$G$\ and $h$ be $\mathcal{E}$-measurable functions satisfying $G(x)\geq
cG(x^{\prime})>0$ for all $x,x^{\prime}\in E$ where $c$ is some finite
positive constant. Let $\{X^{(i)}\}_{1\leq i\leq N}$ be a collection of
independent random samples from $\mu$. If $h$ has finite oscillation then for
any integer $r\geq1$ there exists a finite constant $a_{r}$, independent of
$N$, $G$\ and $h$, such that
\[
\sqrt{N}\mathbb{E}\left\{  \left\vert \frac{\sum_{i=1}^{N}G(X^{(i)}%
)h(X^{(i)})}{\sum_{i=1}^{N}G(X^{(i)})}-\frac{\mu(Gh)}{\mu(G)}\right\vert
^{r}\right\}  ^{\frac{1}{r}}\leq c^{-1}\mbox{\rm osc}(h)a_{r}.
\]

\end{lem}

\begin{proof} The result for $G=1$ and $c=1$ is proved in
\citet{delmoral2004}. The case stated here can be
established using the
representation%
\[
\frac{\mu^{N}(Gh)}{\mu^{N}(G)}-\frac{\mu(Gh)}{\mu(G)}=\frac{\mu(G)}{\mu
^{N}(G)}\left(  \mu^{N}-\mu\right)  \left[  \frac{G}{\mu(G)}\left(
h-\frac{\mu(Gh)}{\mu(G)}\right)  \right]
\]
where $\mu^{N}(dx)=N^{-1}\sum_{i=1}^{N}\delta_{X^{(i)}}(dx)$.
\end{proof}

\begin{rem}
\label{rem:khinchine}For $k\geq0$, let $h_{k-1}^{N}$ be a $\mathcal{F}%
_{k-1}^{N}\ $measurable function satisfying\ $h_{k-1}^{N}\in\mbox{Osc}_{1}%
(\mathcal{X})$\ almost surely. Then Lemma \ref{lem:khinchine} can be invoked
to establish%
\[
\sqrt{N}\mathbb{E}_{\theta}^{y}\left\{  \left\vert \frac{\eta_{\theta,k}%
^{N}(Gh_{k-1}^{N})}{\eta_{\theta,k}^{N}(G)}-\frac{\Phi_{\theta,k}(\eta
_{\theta,k-1}^{N})(Gh_{k-1}^{N})}{\Phi_{\theta,k}(\eta_{\theta,k-1}^{N}%
)(G)}\right\vert ^{r}\right\}  ^{\frac{1}{r}}\leq c^{-1}a_{r}%
\]
where $G$ is defined as in Lemma \ref{lem:khinchine}.
\end{rem}

Lemma \ref{lem:lpErrorFilter0} to Lemma \ref{lem:lpErrorFilter2}\ are a
consequence of Lemma \ref{lem:khinchine} and the estimates in
(\ref{eq:contractionEst}).

\begin{lem}
\label{lem:lpErrorFilter0} For any $r\geq1$ there exist a finite constant
$a_{r}$ such that the following inequality holds for all $\theta$, $y$, $0\leq
k\leq n$ and $\mathcal{F}_{k-1}^{N}\ $measurable function $\varphi_{n}^{N}%
$\ satisfying$\ \varphi_{n}^{N}\in\mbox{Osc}_{1}(\mathcal{X})$ \newline almost
surely,%
\[
\sqrt{N}\mathbb{E}_{\theta}^{y}\left(  \left\vert \Phi_{\theta,k,n}%
(\eta_{\theta,k}^{N})(\varphi_{n}^{N})-\Phi_{\theta,k-1,n}(\eta_{\theta
,k-1}^{N})(\varphi_{n}^{N})\right\vert ^{r}~~\right)  ^{\frac{1}{r}}\leq
a_{r}~b_{\theta,k,n}~\beta\left(  P_{\theta,k,n}\right)  ,
\]
where, by convention $\Phi_{\theta,-1,n}(\eta_{\theta,-1}^{N})=\eta_{\theta
,n}$, and the constants $~b_{\theta,k,n}~$and $\beta\left(  P_{\theta
,k,n}\right)  $ were defined in (\ref{eq:contractionEst}).\newline
\end{lem}

\begin{proof}
\begin{align*}
&  \Phi_{k,n}(\eta_{k}^{N})(\varphi_{n}^{N})-\Phi_{k-1,n}(\eta_{k-1}%
^{N})(\varphi_{n}^{N})\\
&  =\int\left(  \frac{\eta_{k}^{N}(dx_{k})Q_{k,n}(1)(x_{k})}{\eta_{k}%
^{N}Q_{k,n}(1)}-\frac{\Phi_{k}(\eta_{k-1}^{N})(dx_{k})Q_{k,n}(1)(x_{k})}%
{\Phi_{k}(\eta_{k-1}^{N})Q_{k,n}(1)}\right)  P_{k,n}(\varphi_{n}^{N})(x_{k})
\end{align*}
where $\Phi_{0}(\eta_{-1}^{N})=\eta_{0}$ by convention. Applying Lemma
\ref{lem:khinchine} with the estimates in (\ref{eq:contractionEst}) we have%
\[
\sqrt{N}\mathbb{E}\left(  \left\vert \Phi_{k,n}(\eta_{k}^{N})(\varphi_{n}%
^{N})-\Phi_{k-1,n}(\eta_{k-1}^{N})(\varphi_{n}^{N})\right\vert ^{r}~\left\vert
~\mathcal{F}_{k-1}^{N}\right.  \right)  ^{\frac{1}{r}}\leq a_{r}~b_{k,n}%
~\beta\left(  P_{k,n}\right)
\]
almost surely.
\end{proof}

Lemma \ref{lem:lpErrorFilter0}\ may be used to derive the following error
estimate \citep[Theorem
7.4.4]{delmoral2004}.

\begin{lem}
\label{lem:lpErrorFilter}For any $r\geq1$, there exists a constant $c_{r}$
such that the following inequality holds for all $\theta$, $y$, $n\geq0$ and
$\varphi\in\mbox{Osc}_{1}(\mathcal{X})$,
\begin{equation}
\sqrt{N}\mathbb{E}_{\theta}^{y}\left(  \left\vert [\eta_{\theta,n}^{N}%
-\eta_{\theta,n}](\varphi)\right\vert ^{r}\right)  ^{\frac{1}{r}}\leq
c_{r}~\sum_{k=0}^{n}~b_{\theta,k,n}~\beta\left(  P_{\theta,k,n}\right)  .
\label{eq:lpErrorFilter}%
\end{equation}
Assume (A). For any $r\geq1$, there exists a constant $c_{r}^{\prime}$ such
that for all $\theta$, $y$, $n\geq0$, $\varphi\in\mbox{Osc}_{1}(\mathcal{X})$,
$G\in\mathcal{B}(\mathcal{X})$ such that $G$\ is positive and satisfies
$G(x)\geq c_{G}G(x^{\prime})$\ for all $x,x^{\prime}\in\mathcal{X}$ for some
positive constant $c_{G}$,
\begin{equation}
\sqrt{N}\mathbb{E}_{\theta}^{y}\left(  \left\vert \left[  \frac{\eta
_{\theta,n}^{N}(dx_{n})G(x_{n})}{\eta_{\theta,n}^{N}(G)}-\frac{\eta_{\theta
,n}(dx_{n})G(x_{n})}{\eta_{\theta,n}(G)}\right]  \left(  \varphi\right)
\right\vert ^{r}\right)  ^{\frac{1}{r}}\leq c_{r}^{\prime} (1+ c_G^{-1}).
\label{eq:lpErrorFilter1}%
\end{equation}

\end{lem}

\begin{proof} The first part follows from applying Lemma
\ref{lem:lpErrorFilter0} to the telescopic sum \citep[Theorem 7.4.4]%
{delmoral2004}:%
\[
\left(  \eta_{n}^{N}-\eta_{n}\right)  (\varphi)=\sum\limits_{k=0}^{n}%
\Phi_{k,n}(\eta_{k}^{N})(\varphi)-\Phi_{k-1,n}(\eta_{k-1}^{N})(\varphi)
\]
with the convention that $\Phi_{-1,n}(\eta_{-1}^{N})=\eta_{n}$.
For the second part, use the same telescopic sum but with the
$k$-th term being
\begin{align*}
&  \frac{\Phi_{k,n}(\eta_{k}^{N})(\varphi G)}{\Phi_{k,n}(\eta_{k}^{N}%
)(G)}-\frac{\Phi_{k-1,n}(\eta_{k-1}^{N})(\varphi G)}{\Phi_{k-1,n}(\eta
_{k-1}^{N})(G)}\\
&  =\int\left(  \frac{\eta_{k}^{N}(dx_{k})Q_{k,n}(G)(x_{k})}{\eta_{k}%
^{N}Q_{k,n}(G)}-\frac{\Phi_{k}(\eta_{k-1}^{N})(dx_{k})Q_{k,n}(G)(x_{k})}%
{\Phi_{k}(\eta_{k-1}^{N})Q_{k,n}(G)}\right)  \frac{Q_{k,n}(G\varphi)(x_{k}%
)}{Q_{k,n}(G)(x_{k})}.
\end{align*}
Apply Lemma \ref{lem:khinchine} using the same estimates in
(\ref{eq:contractionEst}), i.e. the same estimates hold with $G$\ replacing
$1$ in the definition of $b_{k,n}$ and with $G$ replacing $Q_{q,n}(1)$\ in the
argument of $\beta$.
\end{proof}

The following result is a consequence of Lemma \ref{lem:lpErrorFilter}.

\begin{lem}
\label{lem:lpErrorFilter1}Assume (A). For any $r\geq1$, there exists a
constant $c_{r}$ such that the following inequality holds for all $\theta$,
$y$, $0\leq k\leq n$, $N>0$ and $\varphi_{n}\in\mbox{Osc}_{1}(\mathcal{X})$,
\[
\sqrt{N}\mathbb{E}_{\theta}^{y}\left(  \left\vert \left[  \Phi_{\theta
,k,n}(\eta_{\theta,k}^{N})-\Phi_{\theta,k,n}(\eta_{\theta,k})\right]  \left(
\varphi_{n}\right)  \right\vert ^{r}\right)  ^{\frac{1}{r}}\leq c_{r}%
\overline{\rho}^{n-k}%
\]

\end{lem}

\begin{proof} The result is established by expressing $\Phi_{k,n}(\eta_{k}^{N})$\ as
\[
\Phi_{k,n}(\eta_{k}^{N})(dx_{n})=\int\frac{\eta_{k}^{N}(dx_{k})Q_{k,n}%
(1)(x_{k})}{\eta_{k}^{N}Q_{k,n}(1)}P_{k,n}(x_{k},dx_{n}),
\]
expressing $\Phi_{k,n}(\eta_{k})$ similarly, setting $G$\ in
(\ref{eq:lpErrorFilter1}) to $Q_{k,n}(1)$, $\varphi=P_{k,n}(\varphi_{n})$ and
using the estimates in (\ref{eq:contractionEst}).
\end{proof}

\begin{lem}
\label{lem:lpErrorFilter2} For each $r\geq1$, there exists a finite constant
$c_{r}$ such that for all $\theta$, $y$, $0\leq k\leq q\leq n$, and
$\mathcal{F}_{k-1}^{N}\ $measurable functions $\varphi_{q}^{N}$\ satisfying
$\varphi_{q}\in\mbox{Osc}_{1}(\mathcal{X})$ almost surely,\newline\newline%
\begin{align*}
\sqrt{N}\mathbb{E}_{\theta}^{y}  &  \left(  \left\vert \int\left(  \frac
{\Phi_{\theta,k,q}(\eta_{\theta,k}^{N})(dx_{q})Q_{\theta,q,n}(1)(x_{q})}%
{\Phi_{\theta,k,q}(\eta_{\theta,k}^{N})Q_{\theta,q,n}(1)}-\frac{\Phi
_{\theta,k-1,q}(\eta_{\theta,k-1}^{N})(dx_{q})Q_{\theta,q,n}(1)(x_{q})}%
{\Phi_{\theta,k-1,q}(\eta_{\theta,k-1}^{N})Q_{\theta,q,n}(1)}\right)
\varphi_{q}^{N}(x_{q})\right\vert ^{r}~~\right)  ^{\frac{1}{r}}\\
&  \leq c_{r}~b_{\theta,k,n}~\beta\left(  \frac{Q_{\theta,k,q}(x_{k}%
,dx_{q})Q_{\theta,q,n}(1)(x_{q})}{Q_{\theta,k,q}(Q_{\theta,q,n}(1)) (x_k)}\right)
\end{align*}

\end{lem}

\begin{proof} This results is established by noting that%
\begin{align*}
&  \frac{\Phi_{k,q}(\eta_{k}^{N})(dx_{q})Q_{q,n}(1)(x_{q})}{\Phi_{k,q}%
(\eta_{k}^{N})Q_{q,n}(1)}-\frac{\Phi_{k-1,q}(\eta_{k-1}^{N})(dx_{q}%
)Q_{q,n}(1)(x_{q})}{\Phi_{k-1,q}(\eta_{k-1}^{N})Q_{q,n}(1)}\\
&  =\int\left(  \frac{\eta_{k}^{N}(dx_{k})Q_{k,n}(1)(x_{k})}{\eta_{k}%
^{N}Q_{k,n}(1)}-\frac{\Phi_{k}(\eta_{k-1}^{N})(dx_{k})Q_{k,n}(1)(x_{k})}%
{\Phi_{k}(\eta_{k-1}^{N})Q_{k,n}(1)}\right)  \frac{Q_{k,q}(x_{k}%
,dx_{q})Q_{q,n}(1)(x_{q})}{Q_{k,n}(1)(x_{k})}.
\end{align*}
Now Lemma \ref{lem:khinchine} is applied using the estimates in
(\ref{eq:contractionEst}).
\end{proof}

\begin{lem}
\label{lem:lpErrorDpn} Assume (A). There exists a collection of a pair of
finite positive constants, $a_{i},c_{i}$, $i\geq1$, such that the following
bounds hold for all $r\geq1$, $\theta$, $y$, $0\leq p\leq n$, $N\geq1$,
$x_{p}\in\mathcal{X}$, $F_{p}\in\mathcal{B(X}^{p+1}\mathcal{)}$, $F_{n}%
\in\mathcal{B(X}^{n+1}\mathcal{)}$,\newline%
\begin{align*}
\sqrt{N}\mathbb{E}_{\theta}^{y}\left(  \left\vert \mathcal{M}_{\theta,p}%
^{N}\left(  F_{p}(.,x_{p})\right)  (x_{p})-\mathcal{M}_{\theta,p}\left(
F_{p}(.,x_{p})\right)  (x_{p})\right\vert ^{r}~~\right)  ^{\frac{1}{r}}  &
\leq\left\Vert F_{p}\right\Vert a_{r}p,\\
\sqrt{N}\mathbb{E}_{\theta}^{y}\left(  \left\vert D_{\theta,p,n}^{N}(F_{n})(x_{p}%
)-D_{\theta,p,n}(F_{n})(x_{p})\right\vert ^{r}~~\right)  ^{\frac{1}{r}}  &
\leq a_{r}c_{n}\left\Vert F_{n}\right\Vert .
\end{align*}

\end{lem}

\begin{proof} For each $x_{p}$, let $x_{0:p-1}\rightarrow G_{p-1,x_{p}}(x_{0:p-1}%
)=F_{p}(x_{0:p})q(x_{p}|x_{p-1})$. Adopting the convention $\widetilde{\eta
}_{0}^{N}=\eta_{0}$,%
\begin{align*}
&  \mathcal{M}_{p}^{N}\left(  F_{p}(.,x_{p})\right)  (x_{p})-\mathcal{M}%
_{p}\left(  F_{p}(.,x_{p})\right)  (x_{p})\\
&  =\sum_{k=1}^{p}\int\left(  \frac{\eta_{p-k}^{N}D_{p-k,p-1}^{N}%
(dx_{0:p-1})q(x_{p}|x_{p-1})}{\eta_{p-k}^{N}D_{p-k,p-1}^{N}(q(x_{p}|.))}%
-\frac{\widetilde{\eta}_{p-k}^{N}D_{p-k,p-1}^{N}(dx_{0:p-1})q(x_{p}|x_{p-1}%
)}{\widetilde{\eta}_{p-k}^{N}D_{p-k,p-1}^{N}(q(x_{p}|.))}\right)
F_{p}(x_{0:p})\\
&  =\sum_{k=1}^{p}\int\left(
\frac{\eta_{p-k}^{N}(dx_{p-k}) Q_{p-k,p-1}%
(q(x_{p}|.))(x_{p-k})}{\eta_{p-k}^{N}Q_{p-k,p-1}(q(x_{p}|.))}%
- \frac{\widetilde{\eta}_{p-k}^{N}(dx_{p-k}) Q_{p-k,p-1}%
(q(x_{p}|.))(x_{p-k})}{\widetilde{\eta}_{p-k}^{N}Q_{p-k,p-1}(q(x_{p}|.))}
\right)   \\
& \qquad \qquad \times \frac{G_{p-k,p-1,x_{p}}^{N}(x_{p-k})}{Q_{p-k,p-1}%
(q(x_{p}|.))(x_{p-k})}%
\end{align*}
where $G_{p-k,p-1,x_{p}}^{N}(x_{p-k})=D_{p-k,p-1}^{N}(G_{p-1,x_{p}})(x_{p-k}%
)$, which is a $\mathcal{F}_{p-k-1}^{N}$-measurable function with norm%
\[
\sup_{x_{p-k}}\left\vert \frac{G_{p-k,p-1,x_{p}}^{N}(x_{p-k})}{Q_{p-k,p-1}%
(q(x_{p}|.))(x_{p-k})}\right\vert \leq\left\Vert F_{p}\right\Vert .
\]
The result is established upon applying Lemma \ref{lem:khinchine} (see Remark
\ref{rem:khinchine}) to each term in the sum separately and using the
estimates in (\ref{eq:contractionEst}).
To establish the second result, let
\[F_{p,n}(x_{0:p})=\int Q_{p+1}%
(x_{p},dx_{p+1})\cdots Q_{n}(x_{n-1},dx_{n})F_{n}(x_{0:n}).\] Then,
\[
D_{p,n}^{N}(F_{n})(x_{p})-D_{p,n}(F_{n})(x_{p})=\mathcal{M}_{p}^{N}\left(
F_{p,n}(.,x_{p})\right)  (x_{p})-\mathcal{M}_{p}\left(  F_{p,n}(.,x_{p}%
)\right)  (x_{p}).
\]
The result follows by setting $c_{n}=p\sup_{\theta}\left\Vert Q_{\theta,p,n}%
(1)\right\Vert $ and it follows from Assumption (A) that $c_{n}$ is finite.
\end{proof}

Lemma \ref{theoremLperrorProofPart1} and Lemma \ref{theoremLperrorProofPart2}
both build on the previous results and are needed for the proof of Theorem
\ref{theo:Lperror}.

\begin{lem}
\label{theoremLperrorProofPart1} Assume (A). For any $r\geq1$ there exists a
constant $C_{r}$ such that for all $\theta$, $y$, $0\leq k<n$, $N\geq1$,
$\varphi_{n}\in\mbox{Osc}_{1}(\mathcal{X})$,
\begin{align}
\sqrt{N}\mathbb{E}  &  _{\theta}^{y}\left\{  \left\vert \int \mathbb{Q}_{\theta,n}^{N}(dx_{0:n}) t_{\theta
,k}\left(  x_{k-1},x_{k}\right)  \left(  \varphi_{n}(x_{n})-\eta_{\theta
,n}^{N}(\varphi_{n})\right)  \right.
\right. \nonumber\\
&  \quad\left.  \left.  -\int
\frac{\eta_{\theta,k}^{N}D_{\theta,k,n}^{N}(dx_{0:n})}{\eta_{\theta,k}^{N}%
D_{\theta,k,n}^{N}(1)}
t_{\theta,k}\left(  x_{k-1},x_{k}\right)
\left(  \varphi_{n}(x_{n})-\frac{\eta_{\theta,k}^{N}D_{\theta,k,n}^{N}%
(\varphi_{n})}{\eta_{\theta,k}^{N}D_{\theta,k,n}^{N}(1)}\right)  \right\vert ^{r}\right\}  ^{\frac{1}{r}}\nonumber\\
&  \leq2(n-k)C_{r}\overline{\rho}^{n-k}
\label{eq:filtDervLpError_kthErrorTerm_term1_bound}%
\end{align}
\end{lem}

\begin{proof} The term
(\ref{eq:filtDervLpError_kthErrorTerm_term1_bound}) can be further expanded as
\begin{align}
& \int \frac{\eta_{k}^{N}D_{k,n}^{N}(dx_{0:n})}{\eta_{k}^{N}D_{k,n}%
^{N}(1)} t_{k}\left(  x_{k-1},x_{k}\right)  \left(  \varphi_{n}(x_{n}%
)-\frac{\eta_{k}^{N}D_{k,n}^{N}(\varphi_{n})}{\eta_{k}^{N}D_{k,n}^{N}%
(1)}\right)
\nonumber\\
&\quad-\int \mathbb{Q}_{n}^{N}(dx_{0:n}) t_{k}\left(  x_{k-1},x_{k}\right)  \left(  \varphi_{n}(x_{n})-\eta
_{n}^{N}(\varphi_{n})\right)
  \nonumber\\
& = \sum_{p=k}^{n-1} \int \frac{\eta_{p}^{N}D_{p,n}^{N}(dx_{0:n})}{\eta_{p}%
^{N}D_{p,n}^{N}(1)}
 t_{k}\left(  x_{k-1},x_{k}\right)  \left(  \varphi
_{n}(x_{n})-\frac{\eta_{p}^{N}D_{p,n}^{N}(\varphi_{n})}{\eta_{p}^{N}%
D_{p,n}^{N}(1)}\right)   \nonumber\\
& \quad - \sum_{p=k}^{n-1} \int
\frac{\eta_{p+1}^{N}D_{p+1,n}^{N}(dx_{0:n})}{\eta_{p+1}%
^{N}D_{p+1,n}^{N}(1)}
t_{k}\left(  x_{k-1},x_{k}\right)  \left(  \varphi_{n}(x_{n}%
)-\frac{\eta_{p+1}^{N}D_{p+1,n}^{N}(\varphi_{n})}{\eta_{p+1}^{N}D_{p+1,n}%
^{N}(1)}\right)   \nonumber\\
& = \sum_{p=k}^{n-1}\int
\left(  \frac{\eta_{p}^{N}D_{p,n}^{N}(dx_{0:n})}%
{\eta_{p}^{N}D_{p,n}^{N}(1)}-\frac{\eta_{p+1}^{N}D_{p+1,n}^{N}(dx_{0:n})}%
{\eta_{p+1}^{N}D_{p+1,n}^{N}(1)}\right)
t_{k}\left(  x_{k-1},x_{k}\right)  \left(  \varphi
_{n}(x_{n})-\frac{\eta_{p}^{N}D_{p,n}^{N}(\varphi_{n})}{\eta_{p}^{N}%
D_{p,n}^{N}(1)}\right)
\nonumber\\
& \quad -\sum_{p=k}^{n-1}\left(  \frac{\eta_{p}^{N}D_{p,n}^{N}(\varphi_{n})}%
{\eta_{p}^{N}D_{p,n}^{N}(1)}-\frac{\eta_{p+1}^{N}D_{p+1,n}^{N}(\varphi_{n}%
)}{\eta_{p+1}^{N}D_{p+1,n}^{N}(1)}\right)  \left(  \frac{\eta_{p+1}%
^{N}D_{p+1,n}^{N}(t_{k})}{\eta_{p+1}^{N}D_{p+1,n}^{N}(1)}-\frac{\eta_{p}%
^{N}D_{p,n}^{N}(t_{k})}{\eta_{p}^{N}D_{p,n}^{N}(1)}\right)  \nonumber\\
& \quad -\sum_{p=k}^{n-1}\left(  \frac{\eta_{p}^{N}D_{p,n}^{N}(\varphi_{n})}%
{\eta_{p}^{N}D_{p,n}^{N}(1)}-\frac{\eta_{p+1}^{N}D_{p+1,n}^{N}(\varphi_{n}%
)}{\eta_{p+1}^{N}D_{p+1,n}^{N}(1)}\right)  \frac{\eta_{p}^{N}D_{p,n}^{N}%
(t_{k})}{\eta_{p}^{N}D_{p,n}^{N}(1)}\nonumber\\
\begin{split}
&=\sum_{p=k}^{n-1}\int
\left(  \frac{\eta_{p}^{N}D_{p,n}^{N}(dx_{0:n})}{\eta_{p}%
^{N}D_{p,n}^{N}(1)}-\frac{\eta_{p+1}^{N}D_{p+1,n}^{N}(dx_{0:n})}{\eta
_{p+1}^{N}D_{p+1,n}^{N}(1)}\right) \\
& \qquad\qquad\times
\left(  t_{k}\left(  x_{k-1},x_{k}\right)  -\frac
{\eta_{p}^{N}D_{p,n}^{N}(t_{k})}{\eta_{p}^{N}D_{p,n}^{N}(1)}\right)  \left(
\varphi_{n}(x_{n})-\frac{\eta_{p}^{N}D_{p,n}^{N}(\varphi_{n})}{\eta_{p}%
^{N}D_{p,n}^{N}(1)}\right)
\end{split}
\label{eq:filtDervLpError_kthErrorTerm_term1_subterm1}\\
& \quad -\sum_{p=k}^{n-1}\left(  \frac{\eta_{p}^{N}D_{p,n}^{N}(\varphi_{n})}%
{\eta_{p}^{N}D_{p,n}^{N}(1)}-\frac{\eta_{p+1}^{N}D_{p+1,n}^{N}(\varphi_{n}%
)}{\eta_{p+1}^{N}D_{p+1,n}^{N}(1)}\right)  \left(  \frac{\eta_{p+1}%
^{N}D_{p+1,n}^{N}(t_{k})}{\eta_{p+1}^{N}D_{p+1,n}^{N}(1)}-\frac{\eta_{p}%
^{N}D_{p,n}^{N}(t_{k})}{\eta_{p}^{N}D_{p,n}^{N}(1)}\right)
\label{eq:filtDervLpError_kthErrorTerm_term1_subterm2}%
\end{align}
For the first equality, note that $\eta_{n}^{N}D_{n,n}^{N}(dx_{0:n}%
)=\mathbb{Q}_{n}^{N}(dx_{0:n}).$ It is straightforward to establish that
\begin{equation}
\eta_{p}^{N}D_{p,n}^{N}(dx_{0:n})/\eta_{p}^{N}\left(  g(\left.  y_{p}%
\right\vert \cdot)\right)  =\widetilde{\eta}_{p+1}^{N}D_{p+1,n}^{N}%
(dx_{0:n}),\label{eq:smcLaw_etapDpn_equivalence}%
\end{equation}
which is due to
\begin{align*}
&  \frac{\eta_{p}^{N}(dx_{p})}{\eta_p^N(g(\left. y_p \right\vert \cdot))} \prod\limits_{j=p}^{n-1}Q_{j+1}(x_{j},dx_{j+1})\\
&  =\frac{\eta_{p}^{N}(dx_{p})g(\left.  y_{p}\right\vert x_{p})f(\left.
x_{p+1}\right\vert x_{p})}{\eta_{p}^{N}\left(  g(\left.  y_{p}\right\vert
\cdot)f(\left.  x_{p+1}\right\vert \cdot)\right)  }\frac{dx_{p+1}\eta_{p}^{N}\left(
g(\left.  y_{p}\right\vert \cdot)f(\left.  x_{p+1}\right\vert \cdot)\right)
}{\eta_{p}^{N}\left(  g(\left.  y_{p}\right\vert \cdot)\right)  }%
\prod\limits_{j=p+1}^{n-1}Q_{j+1}(x_{j},dx_{j+1})\\
&  =M_{p+1}^{N}(x_{p+1},dx_{p})\widetilde{\eta}_{p+1}^{N}%
(dx_{p+1})\prod\limits_{j=p+1}^{n-1}Q_{j+1}(x_{j},dx_{j+1}).
\end{align*}
Thus
\begin{dmath}
 \frac{\eta_{p}^{N}D_{p,n}^{N}(dx_{0:p+1},dx_{n})}{\eta_{p}^{N}D_{p,n}%
^{N}(1)}-\frac{\eta_{p+1}^{N}D_{p+1,n}^{N}(dx_{0:p+1},dx_{n})}{\eta_{p+1}%
^{N}D_{p+1,n}^{N}(1)}\nonumber\\
=\frac{\widetilde{\eta}_{p+1}^{N}D_{p+1,n}^{N}(dx_{0:p+1},dx_{n}%
)}{\widetilde{\eta}_{p+1}^{N}D_{p+1,n}^{N}(1)}-\frac{\eta_{p+1}^{N}%
D_{p+1,n}^{N}(dx_{0:p+1},dx_{n})}{\eta_{p+1}^{N}D_{p+1,n}^{N}(1)}\nonumber\\
 =\left(  \frac{\widetilde{\eta
}_{p+1}^{N}(dx_{p+1})Q_{p+1,n}(1)(x_{p+1})}{\widetilde{\eta}_{p+1}%
^{N}Q_{p+1,n}(1)}-\frac{\eta_{p+1}^{N}(dx_{p+1})Q_{p+1,n}(1)(x_{p+1})}%
{\eta_{p+1}^{N}Q_{p+1,n}(1)}\right)
\mathcal{M}_{p+1}^{N}(x_{p+1},dx_{0:p})
\frac{Q_{p+1,n}(x_{p+1},dx_{n}%
)}{Q_{p+1,n}(1)(x_{p+1})}.\label{eq:lawDecomp}%
\end{dmath}
In the first line, variables $x_{p+2:n-1}$\ of the measures $\eta_pD_{p,n}^N(dx_{0:n})$ and
$\eta_{p+1}D_{p+1,n}^N(dx_{0:n})$ are integrated out while the
second line follows from (\ref{eq:smcLaw_etapDpn_equivalence}). Using
(\ref{eq:lawDecomp}), the term
(\ref{eq:filtDervLpError_kthErrorTerm_term1_subterm1}) can be expressed as
\begin{align}
\sum_{p=k}^{n-1} &  \int\left(  \frac{\widetilde{\eta}_{p+1}^{N}%
(dx_{p+1})Q_{p+1,n}(1)(x_{p+1})}{\widetilde{\eta}_{p+1}^{N}Q_{p+1,n}(1)}%
-\frac{\eta_{p+1}^{N}(dx_{p+1})Q_{p+1,n}(1)(x_{p+1})}{\eta_{p+1}^{N}%
Q_{p+1,n}(1)}\right)  \nonumber\\
&  \times P_{p+1,n}\left(  \varphi_{n}-\frac{\eta_{p}^{N}D_{p,n}%
^{N}(\varphi_{n})}{\eta_{p}^{N}D_{p,n}^{N}(1)}\right)  (x_{p+1})\mathcal{M}%
_{p+1}^{N}\left(  t_{k}  -\frac{\widetilde{\eta
}_{p+1}^{N}D_{p+1,n}^{N}(t_{k})}{\widetilde{\eta}_{p+1}^{N}D_{p+1,n}^{N}%
(1)}\right)  (x_{p+1})\nonumber
\end{align}
Note that by (\ref{eq:assAeq2}), \ (\ref{eq:contractionEst}) and
(\ref{eq:contractionEst2}),%
\[
\left\vert P_{p+1,n}\left(  \varphi_{n}-\frac{\eta_{p}^{N}D_{p,n}%
^{N}(\varphi_{n})}{\eta_{p}^{N}D_{p,n}^{N}(1)}\right)  (x_{p+1})\right\vert
\leq\beta\left(  \frac{Q_{p+1,n}(x_{p+1},dx_{n})}{Q_{p+1,n}(1)(x_{p+1}%
)}\right)  ,
\]
\[
\left\vert \mathcal{M}_{p+1}^{N}\left(  t_{k}
-\frac{\widetilde{\eta}_{p+1}^{N}D_{p+1,n}^{N}(t_{k})}{\widetilde{\eta}%
_{p+1}^{N}D_{p+1,n}^{N}(1)}\right)  (x_{p+1})\right\vert \leq C\beta\left(
M_{p+1}^{N}\ldots M_{k+1}^{N}\right)  .
\]
Thus by (\ref{eq:contractionEst})\ and Lemma \ref{lem:lpErrorFilter2}%
, we conclude that there exists a finite constant $C_{r}$ (depending only on
$r$)
\begin{multline}
\sum_{p=k}^{n-1}\sqrt{N}\mathbb{E}\left\{  \left\vert \int\left(
t_{k}\left(  x_{k-1},x_{k}\right)  -\frac{\eta_{p}^{N}D_{p,n}^{N}(t_{k})}%
{\eta_{p}^{N}D_{p,n}^{N}(1)}\right)  \left(  \varphi_{n}(x_{n})-\frac{\eta
_{p}^{N}D_{p,n}^{N}(\varphi_{n})}{\eta_{p}^{N}D_{p,n}^{N}(1)}\right)
\right. \right. \\
\times\left.  \left.
\left(  \frac{\eta_{p}%
^{N}D_{p,n}^{N}(dx_{0:n})}{\eta_{p}^{N}D_{p,n}^{N}(1)}-\frac{\eta_{p+1}%
^{N}D_{p+1,n}^{N}(dx_{0:n})}{\eta_{p+1}^{N}D_{p+1,n}^{N}(1)}\right)
\right\vert ^{r}\right\}  ^{\frac{1}{r}}
  \leq(n-k)C_{r}\overline{\rho}^{n-k}%
\label{eq:filtDervLpError_kthErrorTerm_term1_subterm1_bound}%
\end{multline}
For the term (\ref{eq:filtDervLpError_kthErrorTerm_term1_subterm2}), it
follows from (\ref{eq:lawDecomp})%
\begin{align*}
&  \frac{\eta_{p+1}^{N}D_{p+1,n}^{N}(t_{k})}{\eta_{p+1}^{N}D_{p+1,n}^{N}%
(1)}-\frac{\eta_{p}^{N}D_{p,n}^{N}(t_{k})}{\eta_{p}^{N}D_{p,n}^{N}(1)}\\
&  =\int\frac{\eta_{p+1}^{N}(dx_{p+1})Q_{p+1,n}(1)(x_{p+1})}{\eta_{p+1}%
^{N}Q_{p+1,n}(1)}\left(  \mathcal{M}_{p+1}^{N}\left(  t_{k}\right)
(x_{p+1})-\frac{\widetilde{\eta}_{p+1}^{N}\left(  Q_{p+1,n}(1)\mathcal{M}%
_{p+1}^{N}\left(  t_{k}\right)  \right)  }{\widetilde{\eta}_{p+1}^{N}%
Q_{p+1,n}(1)}\right)  .
\end{align*}
Thus, using (\ref{eq:assAeq2}) and (\ref{eq:contractionEst2}), there exists
some non-random constant $C$ such that the following bound holds almost surely
for all integers $k\leq p<n$, $N$:
\[
\left\vert \frac{\eta_{p+1}^{N}D_{p+1,n}^{N}(t_{k})}{\eta_{p+1}^{N}%
D_{p+1,n}^{N}(1)}-\frac{\eta_{p}^{N}D_{p,n}^{N}(t_{k})}{\eta_{p}^{N}%
D_{p,n}^{N}(1)}\right\vert \leq C\overline{\rho}^{p-k+1}.
\]
Combine this bound with Lemma \ref{lem:lpErrorFilter0} to conclude that there
exists a finite (non-random) constant $C_{r}$ (depending only on $r$) such
that for all integers $k\leq p<n$, $N$:%
\begin{dmath}
\sqrt{N}\mathbb{E}
{
\left\{  \left\vert \left(  \frac{\eta_{p}^{N}D_{p,n}%
^{N}(\varphi_{n})}{\eta_{p}^{N}D_{p,n}^{N}(1)}-\frac{\eta_{p+1}^{N}%
D_{p+1,n}^{N}(\varphi_{n})}{\eta_{p+1}^{N}D_{p+1,n}^{N}(1)}\right)  \left(
\frac{\eta_{p+1}^{N}D_{p+1,n}^{N}(t_{k})}{\eta_{p+1}^{N}D_{p+1,n}^{N}%
(1)}-\frac{\eta_{p}^{N}D_{p,n}^{N}(t_{k})}{\eta_{p}^{N}D_{p,n}^{N}(1)}\right)
\right\vert ^{r}\right \}
}^{\frac{1}{r}}
\leq C_{r}\overline{\rho}^{n-k}%
\label{eq:filtDervLpError_kthErrorTerm_term1_subterm2_bound}%
\end{dmath}
The result now follows from
(\ref{eq:filtDervLpError_kthErrorTerm_term1_subterm1_bound}) and
(\ref{eq:filtDervLpError_kthErrorTerm_term1_subterm2_bound}).
\end{proof}

\begin{lem}
\label{theoremLperrorProofPart2} Assume (A). For any $r\geq1$ there exists a
constant $C_{r}$ such that for all $\theta$, $y$, $0\leq k<n$, $N\geq1$,
$\varphi_{n}\in\mbox{Osc}_{1}(\mathcal{X})$,%
\begin{multline}
\sqrt{N}\mathbb{E}_{\theta}^{y}\left\{  \left\vert
\int \frac{\eta_{\theta,k}^{N}D_{\theta,k,n}^{N}(dx_{0:n})}%
{\eta_{\theta,k}^{N}D_{\theta,k,n}^{N}(1)}
 t_{\theta,k}\left(x_{k-1},x_{k}\right)  \left(  \varphi_{n}(x_{n})-\frac{\eta_{\theta,k}%
^{N}D_{\theta,k,n}^{N}(\varphi_{n})}{\eta_{\theta,k}^{N}D_{\theta,k,n}^{N}%
(1)}\right)  \right.  \right. \\
\left.  \left.  -\int \mathbb{Q}_{\theta,n}(dx_{0:n})
t_{\theta,k}\left(  x_{k-1},x_{k}\right)  \left(
\varphi_{n}(x_{n})-\eta_{\theta,n}(\varphi_{n})\right)
\right\vert ^{r}\right\}  ^{\frac{1}{r}}\leq C_{r}\overline{\rho
}^{n-k} \label{eq:filtDervLpError_kthErrorTerm_term2_bound}%
\end{multline}

\end{lem}

\begin{proof}
\begin{eqnarray}
\lefteqn{\int \frac{\eta_{k}^{N}D_{k,n}^{N}(dx_{0:n})}{\eta_{k}^{N}D_{k,n}%
^{N}(1)}
t_{k}\left(  x_{k-1},x_{k}\right)  \left(  \varphi_{n}(x_{n}%
)-\frac{\eta_{k}^{N}D_{k,n}^{N}(\varphi_{n})}{\eta_{k}^{N}D_{k,n}^{N}%
(1)}\right)   } \nonumber\\
&=&  \int \mathbb{Q}_{n}(dx_{0:n})
t_{k}\left(  x_{k-1},x_{k}\right)  \left(  \varphi_{n}(x_{n}%
)-\eta_{n}(\varphi_{n})\right)   \nonumber\\
&& + \int
\left(  \frac{\eta_{k}^{N}D_{k,n}^{N}%
(dx_{0:n})}{\eta_{k}^{N}D_{k,n}^{N}(1)}-\mathbb{Q}_{n}(dx_{0:n})\right)
t_{k}\left(  x_{k-1},x_{k}\right)  \left(  \varphi_{n}(x_{n}%
)-\eta_{n}(\varphi_{n})\right)
\label{eq:filtDervLpError_kthErrorTerm_term2_subterm1}\\
&& + \left(  \eta_{n}(\varphi_{n})-\frac{\eta_{k}^{N}D_{k,n}^{N}(\varphi_{n}%
)}{\eta_{k}^{N}D_{k,n}^{N}(1)}\right)  \frac{\eta_{k}^{N}D_{k,n}^{N}(t_{k}%
)}{\eta_{k}^{N}D_{k,n}^{N}(1)}%
\label{eq:filtDervLpError_kthErrorTerm_term2_subterm2}%
\end{eqnarray}
To study the errors, term
(\ref{eq:filtDervLpError_kthErrorTerm_term2_subterm1})\ may be decomposed as%
\begin{dmath*}
\int \left(  \frac{\eta_{k}^{N}D_{k,n}^{N}(dx_{0:n}%
)}{\eta_{k}^{N}D_{k,n}^{N}(1)}-\mathbb{Q}_{n}(dx_{0:n})\right)
t_{k}\left(  x_{k-1},x_{k}\right)  \left(  \varphi_{n}(x_{n})-\eta
_{n}(\varphi_{n})\right)   \\
=\sum\limits_{p=0}^{k}\int
\left(  \frac{\eta_{p}%
^{N}D_{p,n}^{N}(dx_{0:n})}{\eta_{p}^{N}D_{p,n}^{N}(1)}-\frac{\widetilde{\eta
}_{p}^{N}D_{p,n}^{N}(dx_{0:n})}{\widetilde{\eta}_{p}^{N}D_{p,n}^{N}%
(1)}\right)
t_{k}\left(  x_{k-1},x_{k}\right)  \left(
\varphi_{n}(x_{n})-\eta_{n}(\varphi_{n})\right)
\end{dmath*}
with the convention that $\widetilde{\eta}_{0}^{N}=\Phi_{0}\left(  \eta
_{-1}^{N}\right)  =\eta_{0}.$ The term corresponding to $p=k$ can be
expressed as%
\[
\int\left(  \frac{\eta_{k}^{N}(dx_{k})Q_{k,n}(1)(x_{k})}{\eta_{k}^{N}%
Q_{k,n}(1)}-\frac{\widetilde{\eta}_{k}^{N}(dx_{k})Q_{k,n}(1)(x_{k}%
)}{\widetilde{\eta}_{k}^{N}Q_{k,n}(1)}\right)  M_{k}^{N}(x_{k},dx_{k-1}%
)t_{k}\left(  x_{k-1},x_{k}\right)  P_{k,n}(\varphi_{n}-\eta_{n}(\varphi
_{n}))(x_{k})
\]
Using Lemma \ref{lem:khinchine} and Remark \ref{rem:khinchine},
\begin{dmath*}
\sqrt{N}\mathbb{E}
{
\left\{  \left\vert \int\left(  \frac{\eta_{k}^{N}%
(dx_{k})Q_{k,n}(1)(x_{k})}{\eta_{k}^{N}Q_{k,n}(1)}-\frac{\widetilde{\eta}%
_{k}^{N}(dx_{k})Q_{k,n}(1)(x_{k})}{\widetilde{\eta}_{k}^{N}Q_{k,n}(1)}\right)
M_{k}^{N}(t_{k})\left(  x_{k}\right)  P_{k,n}(\varphi_{n}-\eta_{n}(\varphi
_{n}))(x_{k})\right\vert ^{r}\right\}
}^{\frac{1}{r}}\\
\leq C_{r}\overline{\rho}^{n-k}%
\end{dmath*}
Similarly, the $p$th term when $p<k$ can be expressed as%
\begin{dmath*}
\int\left(  \frac{\eta_{p}^{N}D_{p,n}^{N}(dx_{0:n})}{\eta_{p}^{N}%
D_{p,n}^{N}(1)}-\frac{\widetilde{\eta}_{p}^{N}D_{p,n}^{N}(dx_{0:n}%
)}{\widetilde{\eta}_{p}^{N}D_{p,n}^{N}(1)}\right)  t_{k}\left(  x_{k-1}%
,x_{k}\right)  \left(  \varphi_{n}(x_{n})-\eta_{n}(\varphi_{n})\right)  \\
=\int\left(  \frac{\Phi_{p,k-1}(\eta_{p}^{N})(dx_{k-1})Q_{k-1,n}%
(1)(x_{k-1})}{\Phi_{p,k-1}(\eta_{p}^{N})Q_{k-1,n}(1)}-\frac{\Phi
_{p,k-1}(\widetilde{\eta}_{p}^{N})(dx_{k-1})Q_{k-1,n}(1)(x_{k-1})}%
{\Phi_{p,k-1}(\widetilde{\eta}_{p}^{N})Q_{k-1,n}(1)}\right)  \\
\times\int\frac{Q_{k}(x_{k-1},dx_{k})Q_{k,n}(1)(x_{k})}{Q_{k-1,n}%
(1)(x_{k-1})}t_{k}\left(  x_{k-1},x_{k}\right)  P_{k,n}\left(  \varphi
_{n}-\eta_{n}(\varphi_{n})\right)  (x_{k})
\end{dmath*}
Using Lemma \ref{lem:lpErrorFilter2} for the outer integral (recall $\Phi_{p,k-1}(\widetilde{\eta}_p^N)
=\Phi_{p-1,k-1}(\eta_{p-1}^N))$,%
\begin{dmath*}
\sqrt{N}\mathbb{E}
{\left\{  \left\vert \int
\left(  \frac{\eta_{p}^{N}D_{p,n}^{N}(dx_{0:n})}{\eta_{p}^{N}D_{p,n}^{N}%
(1)}-\frac{\Phi_{p}\left(  \eta_{p-1}^{N}\right)  D_{p,n}^{N}(dx_{0:n})}%
{\Phi_{p}\left(  \eta_{p-1}^{N}\right)  D_{p,n}^{N}(1)}\right)
t_{k}\left(  x_{k-1}%
,x_{k}\right)  \left(  \varphi_{n}(x_{n})-\eta_{n}(\varphi_{n})\right)
  \right\vert
^{r}\right\}
}^{\frac{1}{r}}\\
\leq C_{r}\overline{\rho}^{n-k}\overline{\rho}^{k-1-p}%
\end{dmath*}
Combining both cases for $p$ yields
\begin{dmath}
\sqrt{N}\mathbb{E}
{
\left\{  \left\vert \int
\left(  \frac{\eta_{k}^{N}D_{k,n}^{N}(dx_{0:n})}{\eta_{k}^{N}D_{k,n}^{N}%
(1)}-\mathbb{Q}_{n}(dx_{0:n})\right)
t_{k}\left(  x_{k-1}%
,x_{k}\right)  \left(  \varphi_{n}(x_{n})-\eta_{n}(\varphi_{n})\right)
  \right\vert ^{r}\right\}
}^{\frac{1}{r}}\\
\leq C_{r}\overline{\rho}^{n-k}\sum_{p=0}^{k-1}\overline{\rho}%
^{k-1-p}+C_{r}\overline{\rho}^{n-k}\leq C_{r}\overline{\rho}^{n-k}\left(
1+\frac{1}{1-\overline{\rho}}\right)
.\label{eq:filtDervLpError_kthErrorTerm_term2_subterm1_bound}%
\end{dmath}
For (\ref{eq:filtDervLpError_kthErrorTerm_term2_subterm2}), Lemma
\ref{lem:lpErrorFilter1} yields the following estimate%
\begin{equation}
\sqrt{N}\mathbb{E}\left\{  \left\vert \left(  \eta_{n}(\varphi_{n})-\frac
{\eta_{k}^{N}D_{k,n}^{N}(\varphi_{n})}{\eta_{k}^{N}D_{k,n}^{N}(1)}\right)
\frac{\eta_{k}^{N}D_{k,n}^{N}(t_{k})}{\eta_{k}^{N}D_{k,n}^{N}(1)}\right\vert
^{r}\right\}  ^{\frac{1}{r}}\leq C_{r}\overline{\rho}^{n-k}%
.\label{eq:filtDervLpError_kthErrorTerm_term2_subterm2_bound}%
\end{equation}
The proof is completed by summing the bounds in
(\ref{eq:filtDervLpError_kthErrorTerm_term2_subterm1_bound}),
(\ref{eq:filtDervLpError_kthErrorTerm_term2_subterm2_bound}) and inflating
constant $C_{r}$ appropriately.
\end{proof}

\subsection{Proof of Theorem \ref{theo:Lperror}}%

\begin{dmath*}
\zeta_{n}^{N}(\varphi_{n})-\zeta_{n}(\varphi_{n})
=\sum\limits_{k=0}^{n}\int \mathbb{Q}_{n}^{N}(dx_{0:n})
t_{k}\left(  x_{k-1},x_{k}\right)  \left(
\varphi_{n}(x_{n})-\eta_{n}^{N}(\varphi_{n})\right)  \\
-\int \mathbb{Q}_{n}(dx_{0:n})
t_{k}\left(  x_{k-1},x_{k}\right)  \left(  \varphi
_{n}(x_{n})-\eta_{n}(\varphi_{n})\right).
\end{dmath*}
To prove the theorem, it will be shown that the error due to the $k$-th term
in this expression is
\begin{multline*}
\sqrt{N}\mathbb{E}\left\{  \left\vert \int \mathbb{Q}_{n}^{N}(dx_{0:n})
t_{k}\left(  x_{k-1},x_{k}\right)  \left(  \varphi_{n}(x_{n})-\eta_{n}^{N}(\varphi_{n})\right)
 \right. \right. \\
- \left. \left. \int \mathbb{Q}_{n}(dx_{0:n})
t_{k}\left(  x_{k-1},x_{k}\right)  \left(
\varphi_{n}(x_{n})-\eta_{n}(\varphi_{n})\right)  %
\right\vert ^{r}\right\}  ^{\frac{1}{r}}
\leq(n-k+1)C_{r}\overline{\rho}^{n-k}%
\end{multline*}
where constant $C_{r}$\ depends only on $r$ and the bounds in Assumption (A)
(through the estimates $\overline{\rho}$ and $\rho^{2}\delta^{2}$\ in
(\ref{eq:contractionEst}) as well as the bounds on the score).
\begin{eqnarray}
&& \int \mathbb{Q}_{n}^{N}(dx_{0:n})
t_{k}\left(  x_{k-1},x_{k}\right)  \left(  \varphi_{n}(x_{n})-\eta
_{n}^{N}(\varphi_{n})\right)
-\int \mathbb{Q}_{n}(dx_{0:n})
t_{k}\left(x_{k-1},x_{k}\right)  \left(  \varphi_{n}(x_{n})-\eta_{n}(\varphi_{n})\right)
\nonumber \\
&&
\begin{split}
= &\int   \mathbb{Q}_{n}^{N}(dx_{0:n})
t_{k}\left(  x_{k-1},x_{k}\right)  \left(  \varphi_{n}(x_{n}%
)-\eta_{n}^{N}(\varphi_{n})\right)
\\
&\quad -\int
\frac{\eta_{k}^{N}D_{k,n}^{N}(dx_{0:n})}{\eta_{k}^{N}D_{k,n}^{N}(1)}
t_{k}\left(  x_{k-1},x_{k}\right)  \left(  \varphi_{n}(x_{n})-\frac{\eta
_{k}^{N}D_{k,n}^{N}(\varphi_{n})}{\eta_{k}^{N}D_{k,n}^{N}(1)}\right)
\label{eq:filtDervLpError_kthErrorTerm_term1}
\end{split} \\
&& \quad
\begin{split}
 & +\int \frac{\eta_{k}^{N}D_{k,n}^{N}(dx_{0:n})}{\eta_{k}^{N}D_{k,n}^{N}(1)}
 t_{k}\left(  x_{k-1},x_{k}\right)  \left(  \varphi_{n}(x_{n}%
)-\frac{\eta_{k}^{N}D_{k,n}^{N}(\varphi_{n})}{\eta_{k}^{N}D_{k,n}^{N}%
(1)}\right)   \\
 & \quad -\int  \mathbb{Q}_{n}(dx_{0:n}) t_{k}\left(  x_{k-1},x_{k}\right)  \left(  \varphi_{n}%
(x_{n})-\eta_{n}(\varphi_{n})\right)
\label{eq:filtDervLpError_kthErrorTerm_term2}%
\end{split}
\end{eqnarray}
The proof is completed by summing the bounds in Lemma
\ref{theoremLperrorProofPart1} for
(\ref{eq:filtDervLpError_kthErrorTerm_term1}) and Lemma
\ref{theoremLperrorProofPart2} for
(\ref{eq:filtDervLpError_kthErrorTerm_term2}) and inflating constant $C_{r}$ appropriately.

\subsection{Proof of Theorem \ref{theo:gradClt}}

The following result which characterizes the asymptotic behavior of the local
sampling errors defined in (\ref{eq:localSampErr}) is proved in \citet[Theorem 9.3.1]{delmoral2004}

\begin{lem}
\label{lem:locErrClt} Let $\{\varphi_{n}\}_{n\geq0}\subset\mathcal{B}%
(\mathcal{X})$. For any $\theta$, $y$, $n\geq0$, the random vector
$(V_{\theta,0}^{N}(\varphi_{0}),\ldots,V_{\theta,n}^{N}(\varphi_{n}))$
converges in law, as $N\rightarrow\infty$, to $(V_{\theta,0}(\varphi
_{0}),\ldots,V_{\theta,n}(\varphi_{n}))$
where $V_{\theta,i}$ is defined in (\ref{eq:localSampErrLimit}).
\end{lem}

The following multivariate fluctuation theorem first proved under slightly
different assumptions in \citet{DDS09} is needed. See
also \citet{DGMO09} for a related study.

\begin{theo}
\label{theo:clt}Assume (A). For any $\theta$, $y$, $n\geq0$, $F_{n}%
\in\mathcal{B}(\mathcal{X}^{n+1})$,$\sqrt{N}\left(  \mathbb{Q}_{\theta,n}%
^{N}-\mathbb{Q}_{\theta,n}\right)  (F_{n})$ converges in law, as
$N\rightarrow\infty$, to the centered Gaussian random variable
\[
\sum_{p=0}^{n}V_{\theta,p}\left(  G_{\theta,p,n}~\frac{D_{\theta,p,n}%
(F_{n}-\mathbb{Q}_{\theta,n}(F_{n}))}{D_{\theta,p,n}(1)}\right)  .
\]
where $V_{\theta,p}$ is defined in (\ref{eq:localSampErrLimit}).
\end{theo}

\begin{proof} Let
\[
\gamma_{n}=\prod\limits_{k=0}^{n-1}\eta_{k}(g(\left.  y_{k}\right\vert .))
\]
and define the unnormalized measure
\[
\Gamma_{n}=\gamma_{n}\mathbb{Q}_{n}.
\]
The corresponding particle approximation is $\Gamma_{n}^{N}=\gamma_{n}%
^{N}\mathbb{Q}_{n}^{N}$\ where $\gamma_{n}^{N}$ $=%
{\textstyle\prod\nolimits_{k=0}^{n-1}}
\eta_{k}^{N}(g(\left.  y_{k}\right\vert .))$. The result is proven by studying
the limit of $\sqrt{N}\left(  \Gamma_{n}^{N}-\Gamma_{n}\right)  $ since \
\[
\lbrack\mathbb{Q}_{n}^{N}-\mathbb{Q}_{n}](F_{n})=\frac{1}{\gamma_{n}^{N}%
}\left[  \Gamma_{n}^{N}-\Gamma_{n}\right]  \left(  F_{n}-\mathbb{Q}_{n}%
(F_{n})\right)  .
\]
Note that Lemma \ref{lem:lpErrorFilter} implies $\gamma_{n}^{N}$ converges
almost surely to $\gamma_{n}$.
The key to studying the limit of $\sqrt{N}\left(  \Gamma_{n}^{N}-\Gamma
_{n}\right)  $ is the decomposition
\[
\sqrt{N}\left[  \Gamma_{n}^{N}-\Gamma_{n}\right]  (F_{n})=\sum_{p=0}^{n}%
\gamma_{p}^{N}~V_{p}^{N}\left(  D_{p,n}(F_{n})\right)  +R_{n}^{N}(F_{n})
\]
where the remainder term is
\[
R_{n}^{N}(F_{n}):=\sum_{p=0}^{n}\gamma_{p}^{N}~V_{p}^{N}\left(  F_{p,n}%
^{N}\right)  \quad\mbox{\rm and the function}\quad F_{p,n}^{N}:=[D_{p,n}%
^{N}-D_{p,n}](F_{n})
\]
By Slutsky's lemma and by the continuous mapping theorem (see
\citet{Van98}) it suffices to show that
$R_{n}^{N}(F_{n})$ converges to $0$, in probability, as
$N\rightarrow\infty$. To prove this, it will be established that
$\mathbb{E}\left( R_{n}^{N}(F_{n})^{2}\right)  $ is
$\mathcal{O}(N^{-1})$.
Since
\[
\mathbb{E}\left\{  R_{n}^{N}(F_{n})^{2}\right\}  =\sum_{p=0}^{n}%
\mathbb{E}\left\{  \left(  \gamma_{p}^{N}~V_{p}^{N}\left(  F_{p,n}^{N}\right)
\right)  ^{2}\right\}  ,
\]
and $\left\vert \gamma_{p}^{N}\right\vert \leq c_{p}$ almost surely, where $c_{p}$ is some
non-random constant which can be derived using (A), it suffices to
prove that $\mathbb{E}\left(  V_{p}^{N}\left(  F_{p,n}^{N}\right)
^{2}\right)  $ is $\mathcal{O}(N^{-1})$.
By expanding the square one arrives at%
\[
\mathbb{E}\left(  V_{p}^{N}\left(  F_{p,n}^{N}\right)  ^{2}~\left\vert
~\mathcal{F}_{p-1}^{N}\right.  \right)  \leq \Phi_{p}\left(  \eta_{p-1}%
^{N}\right)  \left(  \left(  F_{p,n}^{N}\right)  ^{2}\right)  .
\]
By Assumption (A), for any $x_{p-1}\in\mathcal{X}$,
\[
\Phi_{p}\left(  \eta_{p-1}^{N}\right)  \left(  \left(  F_{p,n}^{N}\right)
^{2}\right)  \leq\rho^{2}\int dx_{p}~f(\left.  x_{p}\right\vert x_{p-1}%
)~F_{p,n}^{N}(x_{p})^{2}.
\]
By Lemma \ref{lem:lpErrorDpn}, $\mathbb{E}\left(  V_{p}^{N}\left(  F_{p,n}%
^{N}\right)  ^{2}\right)  $ is\ $\mathcal{O}(N^{-1})$.
\end{proof}

The next lemma is needed to quantify the variance of the particle estimate of
the filter gradient computed using the path-based method. Note that this lemma
does not require the hidden chain to be mixing. We refer the reader to \citet{DeM01}
for a propagation of chaos analysis.

For any $\theta$, $y=\{y_{n}\}_{n\geq0}$, let $\{\mathcal{V}_{\theta
,n}\}_{n\geq0}$ be a sequence of independent centered Gaussian random fields
defined as follows. For any sequence of functions $\{F_{n}\in\mathcal{B}%
(\mathcal{X}^{n+1})\}_{n\geq0}$\ and any $p\geq0$, $\{\mathcal{V}_{\theta
,n}(F_{n})\}_{n=0}^{p}$ is a collection of independent zero-mean Gaussian
random variables with variances given by
\begin{equation}
\mathbb{E}_{\theta}(F_{n}(X_{0:n})^{2}|y_{0:n-1})-\mathbb{E}_{\theta}%
(F_{n}(X_{0:n})|y_{0:n-1})^{2}.
\end{equation}

\begin{lem}
\label{lem:pathClt}Let $\{\delta_{\theta}\}_{\theta\in\Theta}\subset
\lbrack1,\infty)$ and assume $\delta_{\theta}^{-1}\leq g_{\theta}%
(y|x)\leq\delta_{\theta}$\ for all $(x,y,\theta)\in\mathcal{X}\times
\mathcal{Y}\times\Theta$. For any $\theta$, $y$, $n\geq0$, $F_{n}%
\in\mathcal{B}(\mathcal{X}^{n+1})$,$\sqrt{N}\left(  p_{\theta}^{N}%
(dx_{0:n}|y_{0:n-1})-\mathbb{Q}_{\theta,n}\right)  (F_{n})$ converges in law,
as $N\rightarrow\infty$, to the centered Gaussian random variable
\[
\sum_{p=0}^{n}\mathcal{V}_{\theta,p}\left(  G_{\theta,p,n}~F_{\theta
,p,n}\right)  .
\]
where $G_{\theta,p,n}$ was defined in (\ref{DGP}) and \
\[
F_{\theta,p,n}=\mathbb{E}_{\theta}(F(X_{0:n})|x_{0:p},y_{p+1:n-1}%
)-\mathbb{Q}_{\theta,n}(F_{n})
\]

\end{lem}

\subsubsection{Proof of Theorem \ref{theo:gradClt}}

It follows from Algorithm 1 that%
\begin{align}
&  \left(  \zeta_{n}^{N}-\zeta_{n}\right)  (\varphi_{n})\nonumber\\
&  =\mathbb{Q}_{n}^{N}(\varphi_{n}T_{n})-\mathbb{Q}_{n}(\varphi_{n}%
T_{n})+\mathbb{Q}_{n}(\varphi_{n})\mathbb{Q}_{n}(T_{n})-\mathbb{Q}_{n}%
^{N}(\varphi_{n})\mathbb{Q}_{n}^{N}(T_{n}) \label{eq:gradClt_1}%
\end{align}
The second term on the right hand side \ of the equality can be expressed as%
\begin{align}
&  \mathbb{Q}_{n}(\varphi_{n})\mathbb{Q}_{n}(T_{n})-\mathbb{Q}_{n}^{N}%
(\varphi_{n})\mathbb{Q}_{n}^{N}(T_{n})\nonumber\\
&  =\mathbb{Q}_{n}(\varphi_{n}\mathbb{Q}_{n}(T_{n})+\mathbb{Q}_{n}(\varphi
_{n})T_{n})-\mathbb{Q}_{n}^{N}(\varphi_{n}\mathbb{Q}_{n}(T_{n})+\mathbb{Q}%
_{n}(\varphi_{n})T_{n})\nonumber\\
&  \quad +\left(  \mathbb{Q}_{n}^{N}(\varphi_{n})-\mathbb{Q}_{n}(\varphi
_{n})\right)  \left(  \mathbb{Q}_{n}(T_{n})-\mathbb{Q}_{n}^{N}(T_{n})\right)
. \label{eq:gradClt_2}%
\end{align}
Combining the two expressions in (\ref{eq:gradClt_1}) and (\ref{eq:gradClt_2})
gives
\begin{align*}
&  \left(  \zeta_{n}^{N}-\zeta_{n}\right)  (\varphi_{n})\\
&  =\mathbb{Q}_{n}^{N}\left(  \left(  \varphi_{n}-\mathbb{Q}_{n}(\varphi
_{n})\right)  \left(  T_{n}-\mathbb{Q}_{n}(T_{n})\right)  \right) \\
&  -\mathbb{Q}_{n}\left(  \left(  \varphi_{n}-\mathbb{Q}_{n}(\varphi
_{n})\right)  \left(  T_{n}-\mathbb{Q}_{n}(T_{n})\right)  \right) \\
&  +\left(  \mathbb{Q}_{n}^{N}(\varphi_{n})-\mathbb{Q}_{n}(\varphi
_{n})\right)  \left(  \mathbb{Q}_{n}(T_{n})-\mathbb{Q}_{n}^{N}(T_{n})\right)
\end{align*}
Using Lemma \ref{lem:lpErrorFilter} with $r=2$ and Chebyshev's inequality, we
see that $\left(  \mathbb{Q}_{n}^{N}(\varphi_{n})-\mathbb{Q}_{n}(\varphi
_{n})\right)  $\ converges in probability to 0. Theorem \ref{theo:clt} can now
be invoked with Slutsky's theorem to arrive at the stated result in
(\ref{eq:gradClt}).

Moving on to the uniform bound on the variance, let
\begin{align*}
T_{n}-\mathbb{Q}_{n}(T_{n})   & =\sum_{k=0}^{n}\widetilde{t}_{k},\\
\widetilde{t}_{k}    &=t_{k}-\mathbb{Q}_{n}(t_{k}), \\
\widetilde{\varphi}_{n}    &=\varphi_{n}-\mathbb{Q}_{n}(\varphi_{n}).
\end{align*}
Also, the argument of $V_{p}$\ can be expressed as%
\[
\phi_{p}(x_{p})=\frac{Q_{p,n}(1)(x_{p})}{\eta_{p}Q_{p,n}(1)}\sum_{k=0}%
^{n}\frac{D_{p,n}\left(  \widetilde{\varphi}_{n}\widetilde{t}_{k}%
-\mathbb{Q}_{n}\left(  \widetilde{\varphi}_{n}\widetilde{t}_{k}\right)
\right)  (x_{p})}{D_{p,n}(1)(x_{p})}.
\]
It is straightforward to see that $\eta_{p}(\phi_{p})=0$. Therefore the
variance (see (\ref{eq:localSampErrLimit})) now simplifies to
\begin{equation}
\text{var}\sum_{p=0}^{n}V_{p}\left(  G_{p,n}~\frac{D_{p,n}(F_{n}%
-\mathbb{Q}_{n}(F_{n}))}{D_{p,n}(1)}\right)  =\sum_{p=0}^{n}\eta_{p}(\phi
_{p}^{2}). \label{eq:gradClt_3}%
\end{equation}
Consider the function $\phi_{p}$. For $p\leq k-1$,%
\begin{dmath*}
  \frac{D_{p,n}\left(  \widetilde{\varphi}_{n}\widetilde{t}_{k}%
-\mathbb{Q}_{n}\left(  \widetilde{\varphi}_{n}\widetilde{t}_{k}\right)
\right)  (x_{p})}{D_{p,n}(1)(x_{p})}
  =\int\frac{\eta_{p}(dx_{p}^{\prime})Q_{p,n}(1)(x_{p}^{\prime})}{\eta
_{p}Q_{p,n}(1)} \\  \times \int\left(  \frac{Q_{p,k-1}(x_{p},dx_{k-1})Q_{k-1,n}%
(1)(x_{k-1})}{Q_{p,n}(1)(x_{p})}-\frac{Q_{p,k-1}(x_{p}^{\prime},dx_{k-1}%
)Q_{k-1,n}(1)(x_{k-1})}{Q_{p,n}(1)(x_{p}^{\prime})}\right) \\
 \times\int\frac{Q_{k}(x_{k-1},dx_{k})Q_{k,n}(1)(x_{k})}{Q_{k-1,n}%
(1)(x_{k-1})}\widetilde{t}_{k}(x_{k-1},x_{k})P_{k,n}(\widetilde{\varphi}%
_{n})(x_{k}).
\end{dmath*}
Using the estimates in (\ref{eq:assAeq2}) and (\ref{eq:contractionEst}), this
function is bounded by
\begin{equation}
\sup_{x_{p}}\left\vert \frac{D_{p,n}\left(  \widetilde{\varphi}_{n}%
\widetilde{t}_{k}-\mathbb{Q}_{n}\left(  \widetilde{\varphi}_{n}\widetilde
{t}_{k}\right)  \right)  (x_{p})}{D_{p,n}(1)(x_{p})}\right\vert \leq
C\overline{\rho}^{n-1-p} \label{eq:phi_p_bound_part1}%
\end{equation}
for some constant $C$. When $p\geq k$,%
\begin{align*}
&  \frac{D_{p,n}\left(  \widetilde{\varphi}_{n}\widetilde{t}_{k}%
-\mathbb{Q}_{n}\left(  \widetilde{\varphi}_{n}\widetilde{t}_{k}\right)
\right)  (x_{p})}{D_{p,n}(1)(x_{p})}\\
&  =\int\frac{\eta_{p}(dx_{p}^{\prime})Q_{p,n}(1)(x_{p}^{\prime})}{\eta
_{p}Q_{p,n}(1)}\left(  \mathcal{M}_{p}(\widetilde{t}_{k})(x_{p})P_{p,n}%
(\widetilde{\varphi}_{n})(x_{p})-\mathcal{M}_{p}(\widetilde{t}_{k}%
)(x_{p}^{\prime})P_{p,n}(\widetilde{\varphi}_{n})(x_{p}^{\prime})\right)  .
\end{align*}
Again using the estimates in (\ref{eq:assAeq2}), (\ref{eq:contractionEst}) and
(\ref{eq:contractionEst2}),%

\begin{equation}
\sup_{x_{p}}\left\vert \frac{D_{p,n}\left(  \widetilde{\varphi}_{n}%
\widetilde{t}_{k}-\mathbb{Q}_{n}\left(  \widetilde{\varphi}_{n}\widetilde
{t}_{k}\right)  \right)  (x_{p})}{D_{p,n}(1)(x_{p})}\right\vert \leq
C\overline{\rho}^{n-k}. \label{eq:phi_p_bound_part2}%
\end{equation}
Combining (\ref{eq:phi_p_bound_part1}) and (\ref{eq:phi_p_bound_part2}),
\[
\sup_{x_{p}}\left\vert \phi_{p}(x_{p})\right\vert \leq\frac{C\overline{\rho
}^{n-p}}{1-\overline{\rho}}+C\overline{\rho}^{n-p-1}(n-p),
\]
$0\leq p\leq n$. Combining this bound with (\ref{eq:gradClt_3}) will establish
the result.

\bibliographystyle{plainnat}
\bibliography{paper-ref}

\end{document}